\documentclass[12pt,a4paper,twoside, openany]{amsart}
\usepackage{amssymb,amsmath,amsfonts,amsthm,mathrsfs}

\usepackage{longtable,booktabs,colortbl,multicol}

\usepackage{graphicx}

\usepackage{float}

\graphicspath{ {./images/} }

\usepackage{fancyvrb}
\usepackage{comment}
\usepackage[toc,page]{appendix}

\usepackage{cite}

\renewenvironment{thebibliography}[1]{
	\begin{oldthebibliography}{#1}
		\setlength{\parskip}{0ex}
		\setlength{\itemsep}{0ex}
	}
	{
	\end{oldthebibliography}
}

\usepackage{geometry}
\usepackage{textcomp}
\usepackage{listings}
\usepackage[dvipsnames]{xcolor}
\usepackage{filecontents}

\makeatletter

\newcommand\PrologPredicateStyle{}
\newcommand\PrologVarStyle{}
\newcommand\PrologAnonymVarStyle{}
\newcommand\PrologAtomStyle{}
\newcommand\PrologOtherStyle{}
\newcommand\PrologCommentStyle{}

\newif\ifpredicate@prolog@
\newif\ifwithinparens@prolog@

\lst@SaveOutputDef{`_}\underscore@prolog

\newcount\currentchar@prolog

\newcommand\@testChar@prolog%
{%
  \ifnum\lst@mode=\lst@Pmode%
    \detectTypeAndHighlight@prolog%
  \else
    \ifwithinparens@prolog@%
      \detectTypeAndHighlight@prolog%
    \fi
  \fi
  \global\predicate@prolog@false%
}

\newcommand\detectTypeAndHighlight@prolog
{%
  \def\lst@thestyle{\PrologAtomStyle}%
  \ifpredicate@prolog@%
    \def\lst@thestyle{\PrologPredicateStyle}%
  \else
    \expandafter\splitfirstchar@prolog\expandafter{\the\lst@token}%
    \expandafter\ifx\@testChar@prolog\underscore@prolog%
      \ifnum\lst@length=1%
        \let\lst@thestyle\PrologAnonymVarStyle%
      \else
        \let\lst@thestyle\PrologVarStyle%
      \fi
    \else
      \currentchar@prolog=65
      \loop
        \expandafter\ifnum\expandafter`\@testChar@prolog=\currentchar@prolog%
          \let\lst@thestyle\PrologVarStyle%
          \let\iterate\relax
        \fi
        \advance \currentchar@prolog by 1
        \unless\ifnum\currentchar@prolog>90
      \repeat
    \fi
  \fi
}
\newcommand\splitfirstchar@prolog{}
\def\splitfirstchar@prolog#1{\@splitfirstchar@prolog#1\relax}
\newcommand\@splitfirstchar@prolog{}
\def\@splitfirstchar@prolog#1#2\relax{\def\@testChar@prolog{#1}}

\def\beginlstdelim#1#2%
{%
  \def\endlstdelim{\PrologOtherStyle #2\egroup}%
  {\PrologOtherStyle #1}%
  \global\predicate@prolog@false%
  \withinparens@prolog@true%
  \bgroup\aftergroup\endlstdelim%
}

\newcommand\lang@prolog{Prolog-pretty}
\expandafter\lst@NormedDef\expandafter\normlang@prolog%
  \expandafter{\lang@prolog}

\expandafter\expandafter\expandafter\lstdefinelanguage\expandafter%
{\lang@prolog}
{%
  language            = Prolog,
  keywords            = {},      
  showstringspaces    = false,
  alsoother           = @$,
  moredelim           = **[l][\color{teal}]{(}{)},
  MoreSelectCharTable =
    \lst@DefSaveDef{`(}\opparen@prolog{\global\predicate@prolog@true\opparen@prolog},
}

\newcommand\@ddedToOutput@prolog\relax
\lst@AddToHook{Output}{\@ddedToOutput@prolog}

\lst@AddToHook{PreInit}
{%
  \ifx\lst@language\normlang@prolog%
    \let\@ddedToOutput@prolog\@testChar@prolog%
  \fi
}

\lst@AddToHook{DeInit}{\renewcommand\@ddedToOutput@prolog{}}

\makeatother
\definecolor{PrologPredicate}{RGB}{000,031,255}
\definecolor{PrologVar}      {RGB}{024,021,125}
\definecolor{PrologAnonymVar}{RGB}{000,127,000}
\definecolor{PrologAtom}     {RGB}{186,032,032}
\definecolor{PrologComment}  {RGB}{063,128,127}
\definecolor{PrologOther}    {RGB}{000,000,000}

\renewcommand\PrologPredicateStyle{\color{PrologPredicate}}
\renewcommand\PrologVarStyle{\color{PrologVar}}
\renewcommand\PrologAnonymVarStyle{\color{PrologAnonymVar}}
\renewcommand\PrologAtomStyle{\color{PrologAtom}}
\renewcommand\PrologCommentStyle{\itshape\color{PrologComment}}
\renewcommand\PrologOtherStyle{\color{PrologOther}}

\lstdefinestyle{Prolog-pygsty}
{
  language     = Prolog-pretty,
  upquote      = true,
  stringstyle  = \PrologAtomStyle,
  commentstyle = \PrologCommentStyle,
  literate     =
    {:-}{{\PrologOtherStyle :-}}2
    {,}{{\PrologOtherStyle ,}}1
    {.}{{\PrologOtherStyle .}}1
}

\definecolor{burntumber}{rgb}{0.54, 0.2, 0.14}
\definecolor{coolblack}{rgb}{0.0, 0.18, 0.39}
\definecolor{darkterracotta}{rgb}{0.8, 0.31, 0.36}
\definecolor{frenchbeige}{rgb}{0.65, 0.48, 0.36}
\usepackage{hyperref}
\hypersetup{colorlinks=true,
linkcolor=NavyBlue,
    filecolor=OliveGreen,      
    urlcolor=burntumber,
    citecolor=darkterracotta,
    anchorcolor=frenchbeige}

\usepackage{pgfplotstable}
\pgfplotstableset{
col sep=&,
columns={Dim,No,Names,Num2Nilps,NumTorals,Thin},
begin table=\begin{longtable},
end table=\end{longtable},
every head row/.append style={after row=\endhead},
}

\newcommand{\del}{\partial}
\renewcommand{\sl}{\mathfrak{sl}}
\newcommand{\psl}{\mathfrak{psl}}
\newcommand{\OO}{\mathcal{O}}
\newcommand{\CC}{\mathbb{C}}
\newcommand{\Z}{\mathbb{Z}}
\newcommand{\F}{\mathbb{F}}
\newcommand{\ud}{\mathrm{d}}
\newcommand{\DD}{\mathscr{D}}
\newcommand{\TW}{\mathrm{TW}}
\newcommand{\WT}{\mathrm{WT}}
\newcommand{\W}{\mathcal{W}}
\renewcommand{\O}{\mathcal{O}}

\DeclareMathOperator{\GF}{GF}
\DeclareMathOperator{\GL}{GL}
\DeclareMathOperator{\im}{im}
\DeclareMathOperator{\ad}{ad}

\DeclareMathOperator{\Der}{Der}

\DeclareMathOperator{\End}{End}
\DeclareMathOperator{\Char}{char}

\newcommand{\FF} {\mathbb{F} }

\theoremstyle{plain}
\newtheorem{thm}{Theorem}[section]

\newtheorem{lemma}[thm]{Lemma}

\theoremstyle{definition}
\newtheorem{defin}[thm]{Definition}
\newtheorem{ex}[thm]{Example}
\newtheorem{rmk}[thm]{Remark}

\calclayout 

\newcommand{\gl}{{\mathfrak{gl}}}
\newcommand{\fsl}{{\mathfrak{fsl}}}
\newcommand{\s}{{\mathfrak{s}}}

\begin{document}

	\title{A Prolog assisted search for new simple Lie algebras}
	
	\author[Cushing]{David Cushing}
	\address{School of Mathematics, Statistics and Physics, Newcastle University, Newcastle upon Tyne, Great Britain}
	\email{David.Cushing1@newcastle.ac.uk}

	\author[Stagg]{George W. Stagg}
	\address{School of Mathematics, Statistics and Physics, Newcastle University, Newcastle upon Tyne, Great Britain}
	\email{George.Stagg@newcastle.ac.uk}
	
		\author[Stewart]{David I. Stewart}
	\address{School of Mathematics, Statistics and Physics, Newcastle University, Newcastle upon Tyne, Great Britain}
	\email{David.Stewart@newcastle.ac.uk}

	\date{\today}
	
	\begin{abstract}
		We describe some recent computer investigations with the `Constraint Logic Programming over Finite Domains'---CLP(FD)---library in the Prolog programming environment to search for new simple Lie algebras over the field $\GF(2)$ of $2$ elements. Motivated by a paper of Grishkov et. al., we specifically look for those with a \emph{thin decomposition}, and we settle one of their conjectures. We extrapolate from our results the existence of two new infinite families of simple Lie algebras, in addition to finding seven new sporadic examples in dimension $31$. We also better contextualise some previously discovered simple algebras, putting them into families which do not seem to have ever appeared in the literature, and give an updated table of those currently known.
	\end{abstract}
	
	\maketitle
	
	
	\section{Introduction}
	Our principal ambition for this paper is to make invitations to pure mathematicians to consider deploying constraint logic programming systems to assist in research, and to enthusiasts of the logic programming paradigm to consider applying their skills to problems in Lie theory. To this end we narrate a recent adventure searching for new simple Lie algebras over the field $\FF_2=\GF(2)$ of two elements in dialogue with the Prolog programming environment\footnote{Specifically, SWI-Prolog (threaded, 64 bits, version 7.6.4)\cite{prolog}}, using the software library ``Constraint Logic Programming over Finite Domains'', also known as CLP(FD)\cite{clpfd}.
	
	We describe our results. The necessary mathematical background material is available in the preliminaries section.
	
	The result of our interactions with the output of Prolog was the discovery of two previously unknown \emph{infinite} families of \emph{absolutely} simple Lie algebras over $\GF(2)$, and seven further sporadic new examples of Lie algebras of dimension $31$. 
	
	More specifically, for $n\geq 4$ we define Lie algebras $N_1(n)$ and $N_2(n)$, each of dimension $2^n-1$, that have \emph{thin decompositions} in the sense of Grishkov et al.--- \cite[\S3.4]{GGRZ22}. When $n$ is even, each $N_i(n)$ is simple and new for $n\geq 4$, whereas when $n$ is odd, the derived subalgebra $\DD(N_i(n))$ is a simple codimension $1$-subalgebra of $N_i(n)$; we have that while $\DD(N_1(2n+1))$ is probably a known family, $\DD(N_2(2n+1))$ is also new for $n\geq 2$. Let us highlight that the $15$-dimensional algebra $N_1(4)$ is itself new to the literature, whereas $N_2(4)$ had already been found in \cite{Eic10}. We additionally discovered two further simple Lie algebras of dimension $15$, with our own implementation of the methods of \cite{Eic10}, which, together with $N_1(4)$ have been recently discovered independently in \cite{Em22}. (We have labelled these $\mathrm{EM}$ and $\mathrm{ME}$.)
	
	One thing worthy of note is that we needed an invariant to distinguish the algebras we found from those in the literature; to wit, we calculated the number of $2$-nilpotent elements of known and unknown Lie algebras $L$ of a fixed dimension. This calculation amounts to counting the number of points on a rather complicated affine variety (or scheme if one prefers) over $\GF(2)$, which is a fundamental classical problem. What the authors found remarkable was that while well-known routines in Magma, GAP and Sage were relatively easy to use, their runtime was unsatisfactory and produced no progress after hours. On the other hand, using a more complicated input with Prolog calculated and counted the points in seconds.
	
	This article is intended to be self-contained for a reader conversant with linear algebra and finite fields. With that in mind, its structure is as follows. We begin with preliminaries: \S\ref{sec:prologprolog} gives an overview of Prolog and CLP(FD). Then we give some mathematical background:  \S\ref{sec:liealg} gives the definition of a simple Lie algebra over a field and states what is known about those which are finite-dimensional over an algebraically closed field; \S\ref{sec:thinlie} recalls the notion of a thin decomposition for a finite-dimensional simple Lie algebra over $\GF(2)$ and gives a useful plausibly weaker definition of a thin algebra. In \S\ref{sec:search} we describe our deployment of Prolog to find new simple Lie algebras and check uniqueness up to isomorphism. This gave us our seven new $31$-dimensional examples. Furthermore, extrapolating with the help of Prolog's guidance, we produced two new families of Lie algebras; these are explicitly described with generators and relations in \S\ref{sec:new} where we also we prove that they are simple. We give an updated table of the simple Lie algebras known over $\GF(2)$ in \S\ref{sec:bigtable}, where we are able to give more systematic descriptions to some of those already found.
	
	The GitHub repository \url{github.com/cushydom88/simple-lie-algebras} contains the code we used for this article and a library of known simple Lie algebras up to dimension 31, including our new discoveries.
	
	\section{Preliminaries}
		\subsection{Prologue on Prolog}\label{sec:prologprolog}
	The programming language Prolog was first developed in 1972 by Alain Colmerauer and Philippe Roussel\cite{prolog88}. The name is a portmanteau of the French \textit{Programmation en Logique} (programming in logic), and is associated with research in artificial intelligence\footnote{In the sense of traditional \textit{symbolic artificial intellgence} of the late 20th century, rather than the more modern machine-learning approach\cite{GARNELO201917}.} and computational linguistics.
	
	Presently there are several software packages available that implement various forms of the Prolog language. In this article, we provide code for both the commercial ISO compliant SICStus Prolog\cite{sicstus} and the free and open source.
	
	Unlike most programming languages, the logical programming paradigm of Prolog expresses the intent of the programmer in terms of relations between objects and queries on the resulting combined knowledge base, rather than a strictly ordered list of operations. The basis of Prolog's computation and the original source of its power is in the resolution of Horn clauses\cite{horn_1951}, which can be thought of as general statements for any of:
	\begin{enumerate}
		\item[(a)] \textit{facts} (``assume that the statement $x$ holds'');
		\item[(b)] \textit{implications} (``if $p$ holds then also $q$ holds'');
		\item[(c)] \textit{goals} (``show that $u$ holds'').
	\end{enumerate}
	A typical Prolog session involves the programmer loading statements in the form above into the software, then requesting the resolution of a single goal clause. This conversational style of loading code is referred to by Prolog programmers as \textit{consulting}. Once consulting is complete, Prolog aims to return a solution or otherwise return \texttt{false} or \texttt{no} if it determines that no solution exists. Such an approach makes Prolog (and other logic programming languages) particularly useful in the domain of symbolic mathematics and parsing applications.
	
	Constraint Logic Programming over Finite Domains\cite{clpfd}, also known as CLP(FD), is a software package that extends Prolog to be able to efficiently express and solve problems over finite subsets of the integers. Prolog's syntax and goal resolution routines are extended to include the ability to write clauses in terms of constraints and arithmetic relations between variables. Goal resolution is augmented using constraint programming, reducing the domain until a brute-force search for solutions is tractable. The brute-force component of the goal resolution is known as \emph{labelling}; constraint variables in a list are substituted with specific values from their domains and the implications of those substitutions traced through recursively until a solution is found, or the original value is discarded. Various labelling strategies are available depending on the outstanding domains of the unbound constraint variables. This package allows Prolog to be used to efficiently solve Boolean satisfiability problems (SAT) as well as other more general combinatorial problems. For the rest of this article and in demonstrations of Prolog code we will assume the CLP(FD) package has been loaded.
	
	Prolog has been successfully used to create type checkers\cite{lee2015foundations}, automatic theorem provers\cite{stickel1988prolog}, rewriting systems\cite{felty1991logic}, planning and schedulers\cite{schmid2003inductive}, constraint solvers\cite{colmerauer1987opening}, as well as successful in its original intended field of use of natural language processing\cite{pereira2002prolog,Lally2011NaturalLP}. However, Prolog has not found great popularity in the sphere of more general purpose programming\cite{somogyi1995logic}; its most well-known success story outside academia is probably in the form of IBM's Watson\cite{Lally2011NaturalLP}. Prolog powers Watson's natural language processing module and in 2011 it famously won the \textit{Man vs Machine challenge}, broadcast as part of the popular American television game show \textit{Jeopardy!}

	\subsection{Lie algebras}\label{sec:liealg}	
	We assume the reader is familiar with linear algebra and finite fields; probably \cite[\S2 \& \S4]{Cam98} should suffice.
	With that understanding, we give a short introduction to the category of Lie algebras and their representations, which should provide enough background to understand our results. For a more pedagogical treatment, one could read \cite{FH91} for finite-dimensional complex Lie algebras and classification through Dynkin diagrams; to learn about modular Lie algebras---those over fields of positive characteristic---one could consult \cite{SF88}.
	
	\subsubsection{Basic definitions} Let $k$ be a field. Then a Lie algebra $L$ is vector space over $k$ with a multiplication operation \[[\underline{\hspace{.25cm}},\underline{\hspace{.25cm}}]:L\times L\to L;\qquad (x,y)\mapsto [x,y],\]
	which is: antisymmetric (i.e.~$[x,x]=0$ for all $x\in L$); bilinear in both variables (so $(\lambda x+ y,z)=\lambda(x,z)+(y,z)$); and satisfies the Jacobi identity
	\[[x,[y,z]] = [[x,y],z] + [y,[x,z]], \qquad\text{for all }x\in L.\footnote{Compare the equation above with $\frac{\ud}{\ud x}(yz)=\frac{\ud}{\ud x}(y)z+y\frac{\ud}{\ud x}(z)$. One also notices that Lie algebras are not associative in general.}\]
	In this paper, all our Lie algebras will be finite-dimensional over $k$. The prototypical example of a Lie algebra is the set $\End_k(V)$ of $k$-linear maps on a vector space $V$ over $k$ under the operation $[\alpha,\beta](v)=\alpha\circ\beta(v)-\beta\circ\alpha(v)$. Picking a basis of the $n$-dimensional vector space $V$ gives an identification of $\End_k(V)$ with the Lie algebra $\gl_n$ of all $n\times n$ matrices over $k$ under the operation $[x,y]=xy-yx$.
	
	By way of a little history, Lie algebras were first considered by Sophus Lie as infinitesimal neighbourhoods of the identity of smooth transformation groups. In the more general capacity of the above definition, one finds Lie algebras threaded throughout pure mathematics and physics. The case when $k$ has positive characteristic---for example when $k$ is the field $\GF(2)$ of $2$ elements---has relevance mostly to pure mathematics, but such Lie algebras manage to put in an appearance in areas as diverse as the study of finite groups, number theory and geometric invariant theory, to name a few.
	
	\subsubsection{Lie subalgebras and ideals} A subspace $M$ of $L$ is called a \emph{(Lie) subalgebra} if $[x,y]\in M$ whenever $x,y\in M$. A subspace $I$ of $L$ is an an $\emph{ideal}$ of $L$ if $[x,I]\subseteq I$, for all $x\in L$, and we write $I\triangleleft L$; then for any $a\in L$, $a+I=\{a+x\mid x\in I\}$ is a \emph{coset} of $I$ and one checks that the set of all such cosets $L/I$ has the structure of a Lie algebra, via $[a+I,b+I]=[a,b]+I$. If there are no proper ideals of $L$---i.e. the only ideals of $L$ are $0$ and $L$ itself---then we say $L$ is \emph{simple}. One example of an ideal is the \emph{centre} $Z(L)=\{x\in L\mid [x,L]=0\}$; more generally, for any subset $U\subseteq L$, $Z_L(U)=\{x\in L\mid [x,U]=0\}$ is a subalgebra of $L$ called the \emph{centraliser} of $U$ in $L$. 
	
	At least when $L$ is finite-dimensional over $k$, one can associate to $L$ a list of simple Lie algebras that form the building blocks of $L$, called its \emph{components}. It is the statement of the Jordan--Holder theorem that for any sequence of non-trivial simple ideals $I_1\triangleleft L$, $I_2\triangleleft L/I_1$, $I_3\triangleleft L/(I_1+I_2)$ \dots, the complete unordered set $\{I_1,\dots, I_r\}$ is independent of the choices at each stage. This motivates an appetite to classify the simple Lie algebras.
	
	\subsubsection{Simple Lie algebras} Suppose $L$ is of finite dimension $d$ over $k$ and that $k$ is algebraically closed. In the specific case $k=\CC$ the classification of simple Lie algebras is a famous result of Killing (streamlined by Cartan, Weyl and Dynkin). They are in $1$--$1$ correspondence with their \emph{root systems}, and the possibilities for those root systems are as follows: $A_n$ $(n\geq 1)$, $B_n$ ($n\geq 2$), $C_n$ ($n\geq 3$), $D_n$ ($n\geq 4$), $G_2$, $F_4$, $E_6$, $E_7$ or $E_8$. Each of those root systems has an associated Dynkin diagram and the Lie algebras of the first $4$ families have straightforward descriptions in terms of matrix algebras; for example $A_n$ corresponds to the Lie algebra $\sl_{n+1}$ of traceless $(n+1)\times (n+1)$ matrices.
	
	The modular case, i.e. when the characteristic of $k$ is a prime $p$, is much more complicated. Since the Lie algebras above can be defined over $\Z$, they can be reduced mod $p$ to give Lie algebras defined over $k$, and they are largely still simple.\footnote{An example of the sort of issue that occurs is that if $p\mid n+1$ then the identity matrix $I$ in $\gl_{n+1}$ is of trace $n+1=0\pmod{p}$, hence lies in $\sl_{n+1}$; indeed it spans the centre $Z(\sl_{n+1})$. For $n>1$, however $\psl_{n+1}:=\sl_{n+1}/\langle I\rangle$, is simple.}  But there are many more. For example, the Lie algebra $W_1$ of derivations $\Der(\OO_1)$ of the truncated polynomial algebra $\OO_1:=F[x]/x^p$ is $p$-dimensional, having basis $\{\del,x\del,\dots,x^{p-1}\del\}$. Here $\del$ acts by differentiation and one can check that the Lie bracket is defined through $[x^i\del, x^j\del]=(j-i)x^{i+j-1}\del$. (We treat $x^i\del=0$ if $i>p-1$.) It turns out $W_1$ is simple if $p>2$, $W_1\cong\sl_2$ if $p=3$ and otherwise $W_1$ is not related to those of Dynkin type. 
	
	The good news is that the classification of $L$ up to isomorphism when $p\geq 5$ is available, due to Premet--Strade, and was announced in \cite{PS06}. The books \cite{SF88}, \cite{Str04}, \cite{Str09} and \cite{Str13} give this theory a lucid exposition. A rough description is as follows. For $\underline{n}=(n_1,\dots, n_m)$ a sequence of positive integers, one generalises $\OO_1$ to a divided power algebra $\OO(m;\underline{n})$ of dimension $p^{\sum n_i}$ in $m$ variables. Then $W_1$ is generalised to the  $mp^{\sum n_i}$-dimensional special derivations $W(m;\underline{n})$ of $\OO(m;\underline{n})$; these are the \emph{Jacobson--Witt} Lie algebras and are simple with a small list of exceptions in characteristics $2$ or $3$. Taking the subalgebras of elements of $W(m;\underline{n})$ stabilising certain differential forms gives the other algebras of \emph{Cartan type}: one gets additionally those of type $H$ (Hamiltonian), $S$ (Special) or $K$ (Contact). Lastly in characteristic $5$ one finds Melykian algebras of type $M$. Note that the exact classification in the case of the Hamiltonians is rather involved---\cite{Skr19}. Much of the machinery of the classification breaks down in characteristics $2$ and $3$, and while those known in characteristic $3$ sit in infinite families, the results we have in characteristic $2$ are at present rather  sporadic.
	
	Since we need the notion anyway, we mention that most of the algebras of Cartan type constructed above are not themselves simple and one needs something slightly smaller. It follows from the Jacobi identity that the span of the brackets of elements of $L$ is an ideal $L^{[1]}=[L,L]$ called the \emph{derived subalgebra}. In the cases above either $L^{[1]}$ or its second derived subalgebra $L^{[2]}$ is simple. If $L=[L,L]$ we say that $L$ is \emph{perfect}.
	
	\subsubsection{Homomorphisms and representations}\label{homs} If $L$ and $L'$ are Lie algebras over the field $k$ and $\phi:L\to L'$ a linear map satisfying $\phi([x,y])=[\phi(x),\phi(y)]$ then we say $\phi$ is a \emph{homomorphism} of Lie algebras. It is easily checked that $\ker(\phi)$ is an ideal of $L$ and that $\im(\phi)$ is a subalgebra of $L'$; if $\ker(\phi)=0$ then $\phi$ is called an \emph{embedding}, and if $\im(L)=L'$ then $L$ is a \emph{quotient} map. If $\phi$ is invertible then it is an \emph{isomorphism} and we write $L\cong L'$; for example, when $\phi:L\to L'$ is a homomorphism, one has $L/\ker(\phi)\cong\im(\phi)$. An isomorphism from $L$ to $L$ is called an \emph{automorphism}. (The group of all automorphisms is a useful invariant of $L$.)
	
	In the special case $L'=\End(V)$ for $V$ a $k$-vector space, then we say $V$ is a \emph{representation of $L$}, or an \emph{$L$-module} (in which case we suppress mention of $\phi$); if $\ker(\phi)=0$ then $V$ is \emph{faithful}, whereas if $\ker(\phi)=L$ then $V$ is \emph{trivial}. An $L$-stable subspace $U\subseteq V$ is called a submodule, and the quotient $V/U$ of vector spaces inherits the structure of another $L$-module. We say $V$ is \emph{irreducible} if $0$ and $V$ are its only submodules; thus the $1$-dimensional trivial module $k$ is irreducible. There is always a canonical representation of $L$ called the \emph{adjoint representation}, $\ad:L\to\End(L)$; $x\mapsto \ad(x)$, where $\ad(x)y=[x,y]$; the Jacobi identity implies this is a homomorphism. The kernel of $\ad$ identifies with the centre $Z(L)$; so in particular if $L$ is simple, then $\ad$ is faithful and we have $L\cong \ad(L)$, meaning we can regard $L$ as a subalgebra of $\End(L)$.
	
%
	
	\subsubsection{Derivations and $p$-envelopes} The \emph{algebra of derivations} of $L$ is the Lie subalgebra of $\End(L)$ defined by $\Der(L)=\{d\in\End(L)\mid d([x,y])=[d(x),y]+[x,d(y)]\}$. By Jacobi, each $\ad(x)$ for $x\in L$ is a derivation we call \emph{inner}; equally by Jacobi, one has $\ad(L)\subseteq\Der(L)$ an ideal. For simple complex Lie algebras all derivations are inner, but this is far from true for modular Lie algebras. Indeed if $\Char k=p>0$ and $x\in L$, then one can check that $\ad(x)^p$ is also a derivation of $L$. If for each $x\in L$, $\ad(x)^p=\ad(x^{[p]})$ for some $x^{[p]}\in L$, then $L$ is \emph{restricted}; the algebras $\Der(L)\subseteq\End(L)$ are restricted under the map $d\mapsto d^p$. Assume $L\cong\ad(L)$. Then recursively adding in the span of the $p$-th powers to $\ad(L)$ must eventually stabilise in a restricted subalgebra $\mathscr{L}$ of $\End(V)$, called the \emph{minimal $p$-envelope} of $L$. One often has $\mathscr{L}=\Der(L)$, but in general $\mathscr{L}$ is a proper subalgebra of $\Der(L)$. 
	
	On the other hand, a representation $\rho:\mathscr{L}\to\End(V)$ is \emph{restricted} if for all $x\in\mathscr{L}$ one has $\rho(x^{[p]})=\rho(x)^p$. A representation of $\rho:L\to\End(V)$ is called restricted if it factors through a restricted representation of $\mathscr{L}$. See \cite[\S2]{SF88} for a far less laconic treatment.
	
	\subsubsection{Tori and roots}\label{sec:tor}Keep the assumption $L=\ad(L)\subseteq \gl(L)$ and $\Char k=p>0$. If $x^p=0$, (resp.~$x^{p^i}=0$ for some $i$) we say $x$ is \emph{$p$-nilpotent}, (resp.~\emph{nilpotent}) and if $x^p=x$ we say $x$ is \emph{toral}. A subalgebra $T$ of $L$ with a basis of toral elements is automatically commutative and is called a \emph{torus}; the \emph{rank} of $T$ is its dimension. Usually one is interested in $T$ of maximal dimension in $\mathscr{L}$, in which case the integer $\mathrm{TR}(L):=\dim T$ is called the \emph{absolute toral rank of $L$}. Any restricted $T$-module has a basis of simultaneous eigenvectors for $T$. If $v$ is one such, then $t(v)=\lambda(t)v$ for some linear map $\lambda:T\to \F_p$, which we call a \emph{weight}; and we say $v$ is a \emph{weight vector}. The set of all $v\in V$ such that $t(v)=\lambda(v)$ is the \emph{$\lambda$-weight space} of $V$. In particular, the non-zero weights of $T$ on $L$ are called \emph{roots} and we denote by $\Phi$ the set of these; a weight vector for this action of weight $\alpha\in \Phi$ is a \emph{root vector}; a weight space for $\alpha\in\Phi$ is a \emph{root space}. Crucially, if $e_\alpha$ and $e_\beta$ are root vectors corresponding to roots $\alpha$ and $\beta$, we have $[t,[e_\alpha,e_\beta]]=[[t,e_\alpha],e_\beta]+[e_\alpha,[t,e_\beta]]=(\alpha+\beta)(t)[e_\alpha,e_\beta]$ so that `roots add under the Lie bracket'.

\subsection{Toral switching}\label{sec:torswitchthy} Over $\mathbb C$, maximal diagonalisable subalgebras of a simple Lie algebra $L$ are \emph{Cartan subalgebras}, and these are all conjugate under the isomorphism group of $L$, and in particular are all of the same dimension. This statement is far from true in characteristic $p$; however toral switching is an effective replacement, and it provides a very powerful method to generate isomorphisms between Lie algebras. The basic idea is to take a torus $T$ of maximum dimension in $\mathscr{L}$, a root $\alpha$ of $T$ on $\mathscr{L}$, and a $p$-nilpotent root element $x\in L_\alpha$. Then $\{t+\alpha(t)x\mid t\in T\}$ turns out to be another torus of maximal dimension in $\mathscr{L}$. In full generality over algebraically closed fields, Premet \cite{Pre89} has shown that one may get from any torus of maximum dimension in $\mathscr{L}$ to any other after a finite set of switches; see \cite[\S1.5]{Str04} for more detail.  We use toral switching in a more limited way to create in Prolog a graph of toral switching isomorphisms between certain sets of data generating Lie algebras. See \S\ref{sec:tor} below.
	
	\subsection{Thin Lie algebras}\label{sec:thinlie} For the time being, we let $k=\GF(2)$, the field of two elements.
	
	Following \cite[\S3.4]{GGRZ22}, suppose $L$ is a Lie algebra over $\GF(2)$ of dimension $2^n-1$ with absolute toral rank $\mathrm{TR}(L)=n$. Take a maximal torus $T\subset\mathscr{L}$, and assume it has the set of roots $\Phi=\GF(2)^n\setminus\{\underline{0}\}$ on $L$. Then each root space is $1$-dimensional, $Z_L(T)=0$ and in particular, $T\cap L=0$. In this scenario, we say $L$ has a \emph{thin decomposition}. Note that  by the comments in \S\ref{sec:tor}, if $e_\alpha$ and $e_\beta$ span the root spaces corresponding to roots $\alpha$ and $\beta$ respectively then we have: \begin{equation}\text{either }[e_\alpha,e_\beta]=e_{\alpha+\beta}\text{ or }[e_\alpha,e_\beta]=0.\label{thineqn}\end{equation}  On the other hand, we say a $(2^n-1)$-dimensional Lie algebra $L$ over $\GF(2)$ is \emph{thin} if the weaker condition holds that $L$ is spanned by $2^n-1$ vectors $e_\alpha$ with $\alpha\in\Phi$ satisfying (\ref{thineqn}). Indeed, suppose $L$ is thin and $\{\alpha_i\mid 1\leq i\leq n\}$ the standard basis for $\GF(2)^n$. Then (\ref{thineqn}) implies that the endomorphisms $d_i$ of $L$ defined on $e_\alpha$ via $d_i(e_\alpha)=(\alpha,\alpha_i) e_{\alpha}$ give a commuting set of $n$ linearly independent toral elements of $\Der(L)$; if $\Der(L)=\mathscr{L}$ then the $d_i$ would exhibit a torus giving a thin decomposition. 
	
	To record:
	\begin{lemma}\label{thinvsthin} Let $L$ be a Lie algebra over $\GF(2)$ of dimension $2^n-1$. Then $L$ is thin if and only if there is an $n$-dimensional torus $T\subset\Der(L)$ such that $L$ decomposes for $T$ with root system $\Phi=\GF(2)^n\setminus\{\underline{0}\}$.\end{lemma}

\begin{rmk}It is Conjecture 3 in \cite[\S7]{GGRZ22} that all simple Lie algebras with a thin decomposition have the property that $\Der(L)$ is a semidirect product $T\ltimes L$; i.e. that $\Der(L)=T\oplus L$ as vector spaces and with bracket $[(t,x),(s,y)]=([t,s],t(y)-s(x)+[x,y])$. Unfortunately, the $15$-dimensional algebra denoted `Kap$_2(4)$' which is isomorphic to `Kap$_3(6)$' has a thin decomposition, but its derivation algebra is dimension $20$.\end{rmk}
	
	A surprisingly high proportion of the simple Lie algebras of dimension $2^n-1$ are thin, and it is hypothesised that this explains the spikes in isomorphism classes of simple Lie algebras at dimension $2^n-1$.
	
	There is a significantly smaller amount of data needed to define a thin Lie algebra than a general one.
	\begin{defin}\label{thintable} Take a matrix $T$ of size $(2^n-1)\times (2^n-1)$ with columns and rows labelled by the roots $\alpha\in\Phi$ and with entries $T_{\alpha,\beta}\in \Z/2$. Then $T$ is called a \emph{thin table} if the bracket $[e_\alpha,e_\beta]=T_{\alpha,\beta}e_{\alpha+\beta}$ defines a Lie algebra on the $(2^n-1)$-dimensional vector space with basis $\{e_\alpha\mid\alpha\in\Phi\}$.\end{defin}
	
	\begin{ex}\label{fsl2}The following is a particularly important thin table
		\begin{center}\begin{tabular}{c|c|c|c}
				& $\alpha$ & $\beta$ & $\alpha+\beta$\\
				\hline
				$\alpha$ & $0$ & $1$ & $1$\\
				$\beta$ & $1$ & $0$ & $1$\\
				$\alpha+\beta$ & $1$ & $1$ & $0$\end{tabular},\end{center}
		where $\alpha=(1,0)$, $\beta=(0,1)$ and $\alpha+\beta=(1,1)$ in $\GF(2)^2\setminus\{(0,0)\}$. This describes a $3$-dimensional simple Lie algebra $\s$ with basis $x,y,z$ such that $[x,y]=z$, $[z,x]=y$ and $[y,z]=x$. It is easy to see it is simple. It is in fact the unique simple Lie algebra of dimension $3$ over any field of characteristic $2$. In characteristics other than $2$, this role is played by the algebra $\sl_2$ of $2\times 2$ traceless matrices which is not simple in characteristic $2$. This motivates the name `fake $\sl_2$' for $\s$, henceforth $\fsl_2$.   \end{ex}
	
	\begin{rmk} Another useful basis of $\fsl_2$ is $\{\del,x\del,x^{(2)}\del\}$, which spans the Lie derived subalgebra $W(1;(2))'$ of the special derivations $W(1;(2))$ of the divided power algebra $\OO(1,(2))=\langle 1,x,x^{(2)},x^{(3)}\rangle$. One has $[\del,x\del]=\del$, $[\del,x^{(2)}\del]=x\del$ and $[x\del,x^{(2)}\del]=x^{(2)}\del$. In particular, the element $x\del$ is toral. Since $\ad(\del)^2(x^{(2)}\del)=\del$, the derivation $\ad(\del)^2$ is not inner, thus $\fsl_2$ is not restricted. Letting $F:=\ad(\del)^2$ and $E:=\ad(x^{(2)}\del)^2$, one checks $\langle F,\del,x\del,x^{(2)}\del,E\rangle$ is a minimal $p$-envelope; and $\langle F,x\del,E\rangle$ a subalgebra isomorphic to $\sl_2$.\end{rmk}
	
	For any root $\alpha\in\Phi$ we can decompose $L$ as $L=L_0\oplus L_1$ where $L_i$ is spanned by the $e_\beta$ such that the inner product of $\beta$ with $\alpha$ in $\GF(2)^n$ is $i$. Taking $\alpha=(1,0)$ in Example \ref{fsl2}, we get $L_0=\langle e_{(0,1)}\rangle$ and $L_1=\langle e_{(1,0)},e_{(1,1)}\rangle$. Since `roots add under the Lie bracket', we have $[L_0,L_0]\subseteq L_0$, $[L_1,L_1]\subseteq L_0$ and $[L_0,L_1]\subseteq L_1$. A decomposition $L=L_0\oplus L_1$ satisfying these constraints is called a \emph{$\Z/2$-grading}. The following result gives rise to very strong constraints to feed into Prolog.
	
	\begin{lemma}\label{l0l1} Suppose $L\cong L_0\oplus L_1$ is a $\Z/2$-grading of a simple Lie algebra $L$ with $L_1\neq 0$. Then $[L_1,L_1]=L_0$ and $L_0$ acts faithfully on $L_1$.\end{lemma}
	\begin{proof}From $[L_0,L_1]\subseteq L_1$ and $[L_1,L_1]\subseteq L_0$ one concludes that $[L_1,L_1]+L_1$ is an ideal; as $L$ is simple and $L_1\neq 0$, we therefore deduce $L_0=[L_1,L_1]$. If $0\neq y\in L_0$ centralises $L_1$ then so does $y+[y,L_0]$, by the Jacobi identity. Therefore the sum $ky+[y,L_0]+[[y,L_0],L_0]+\dots$ stabilises in an ideal of $L_0$ centralising $L_1$, hence a proper ideal of $L$. This contradicts the simplicity of $L$.\end{proof}

	\section{Searching for simple thin Lie algebras with Prolog}\label{sec:search}
Here we describe the code used to generate the results in this paper. Since Prolog's declarative code contrasts significantly with more well-known imperative languages such as C, Python, GAP and so on, we introduce it slowly:
\begin{enumerate}\item In \S\ref{sec:toy} we give a illustrative toy version of our code that we hope to be pedagogically expedient; this outputs all thin tables in dimension $7$, most of which are not simple.
\item In \S\ref{sec:simpcheck} we explain how to use the MeatAxe in GAP \cite{GAP} to discard those tables describing algebras which are not simple. (In the full version of the Prolog code, we end up deploying enough constraints that all the tables output are found to be simple in GAP.)
\item In \S\ref{sec:2nilptor} we describe the Prolog code we use to count the number of $2$-nilpotent and toral elements in the Lie algebras we generate. Since these numbers are invariant under isomorphism, Lie algebras with differing numbers of toral or $2$-nilpotent elements must be non-isomorphic.
\item In the next few sections, \S\ref{sec:CushCode}--\ref{sec:tor}, we talk through the full version of the Prolog code. This is coordinated by a master predicate {\tt reduced\_thin\_search} which first invokes an initial search. This is followed by post-processing to remove some non-simple tables. Next, symmetry-breaking through lexicographic ordering refines the list by pruning isomorphism classes\footnote{Suppose $\mathcal{S}$ is a set of solutions to a set of constraints $\mathcal{C}$ and let $\sim$ be an equivalence relation among the elements of $\mathcal{S}$; for example $\sim$ may partition $\mathcal{S}$ into subsets  which describe isomorphic stuctures. Let $\mathcal{S'}\subseteq \mathcal{S}$ satisfy an additional constraint $C$. Then $C$ is said to \emph{break symmetry} [of the relation $\sim$] if the number of equivalence classes of $\mathcal{S}'$ and $\mathcal{S}$ under $\sim$ are the same, but $\mathcal{S}'$ is smaller than $\mathcal{S}$.}. Finally, the technique of toral switching is used to collect these into a very small number of sets of isomorphic algebras.
\item In almost all cases the sets above can be shown to contain pairwise distinct Lie algebras by counting $2$-nilpotent and toral elements. In just a couple of cases in each dimension there is some bespoke work needed to finish the calculation and we describe this in \S\ref{sec:finalfinal}.
\end{enumerate}

	\subsection{Creating the search: toy version}\label{sec:toy}
Let us start by illustrating some more primitive code to generate \emph{all} thin tables in case $\dim L=7$, most of which will not describe simple Lie algebras. 
	
	Prolog's CLP(FD) library will operate on collections of variables and constraints. We set up the relevant variables according to the following $7\times 7$ table, $T$:
	\begin{center}
		\begin{tabular}{ c|c|c|c|c|c|c|c| } 
			& ${001}$ & ${010}$ & ${011}$ & ${100}$ & ${101}$ & ${110}$ & ${111}$ 
			\\ 
			\hline
			${001}$ & $A$ & $B$ & $C$ & $D$ & $E$ & $F$ & $G$ 
			\\ 
			\hline
			${010}$ & $H$ & $I$ & $J$ & $K$ & $L$ & $M$ & $N$ 
			\\ 
			\hline
			${011}$ & $O$ & $P$ & $Q$ & $R$ & $S$ & $T$ & $U$  
			\\ 
			\hline
			${100}$ & $V$ & $W$ & $X$ & $Y$ & $Z$ & $AA$ & $AB$  
			\\ 
			\hline
			${101}$ & $AC$ & $AD$ & $AE$ & $AF$ & $AG$ & $AH$ & $AI$  
			\\ 
			\hline
			${110}$ & $AJ$ & $AK$ & $AL$ & $AM$ & $AN$ & $AO$ & $AP$ 
			\\ 
			\hline
			${111}$ & $AQ$ & $AR$ & $AS$ & $AT$ & $AU$ & $AV$ & $AW$ 
			\\ 
			\hline
		\end{tabular}
	\end{center}
	and each variable takes a value in $\GF(2)$. (Note that Prolog's syntax demands that variables start with capital letters.) The constraints necessary for $T$ to be a thin table (see Definition \ref{thintable}) are as follows:
	
	\subsubsection{Zero on the diagonals}
	For all $\alpha$ we have $[e_\alpha, e_\alpha] = 0.$ Namely,
	$$A = I = Q = Y = AG = AO = AW = 0.$$
	
	\subsubsection{Symmetry of the table}
	For all $\alpha, \beta\in\Phi$, we have $[e_\alpha, e_\beta] = [e_\beta, e_\alpha]$. For example $B=H.$
	
	\subsubsection{The Jacobi identity}
	For all $\alpha, \beta, \gamma$ we have
	$$[e_\alpha,[e_\beta,e_\gamma]] + [e_\beta,[e_\gamma,e_\alpha]] + [e_\gamma,[e_\alpha,e_\beta]]  = 0.$$
	For example, taking $\alpha = (0,0,1), \beta = (0,1,0)$ and $\gamma = (1,0,0)$, we get
	$$F\cdot K + L\cdot D + R\cdot B = 0.$$

Then copying and pasting the code from Figure \ref{code7} into the command line of SWI-Prolog yields a total of $2781$ possible thin tables in dimension $7$.

\begin{figure}
\begin{center}
\begin{lstlisting}
?- Vars=[A, B, C, D, E, F, G, H, I, J, K, L, M, N, O, P, Q, R, S, T, U, 
V, W, X, Y, Z, AA, AB, AC, AD, AE, AF, AG, AH, AI, AJ, AK, AL, AM, AN,
AO, AP, AQ, AR, AS, AT, AU, AV, AW], Vars ins 0..1,
		
% Zero on the diagonals:
A #= 0, I #= 0, Q #= 0, Y #= 0, AG #= 0, AO #= 0, AW #= 0,
		
% Symmetry of the table:
H #= B, O #= C, V #= D, AC #= E, AJ #= F, AQ #= G, P #= J, W #= K, 
AD #= L, AK #= M, AR #= N, X #= R, AE #= S, AL #= T, AS #= U, AF #= Z, 
AM #= AA, AT #= AB, AN #= AH, AU #= AI, AV #= AP,
		
% The Jacobi identity:
((B*Q + J*A + O*I) mod 2) #= 0, ((B*R + K*AJ + V*AD) mod 2) #= 0,
((B*S + L*AQ + AC*W) mod 2) #= 0, ((B*T + M*V + AJ*AR) mod 2) #= 0,
((B*U + N*AC + AQ*AK) mod 2) #= 0, ((C*K + R*AQ + V*AE) mod 2) #= 0,
((C*L + S*AJ + AC*X) mod 2) #= 0, ((C*M + T*AC + AJ*AS) mod 2) #= 0,
((C*N + U*V + AQ*AL) mod 2) #= 0, ((D*AG + Z*A + AC*Y) mod 2) #= 0,
((D*AH + AA*H + AJ*AT) mod 2) #= 0, ((D*AI + AB*O + AQ*AM) mod 2) #= 0,
((E*AA + AH*O + AJ*AU) mod 2) #= 0, ((E*AB + AI*H + AQ*AN) mod 2) #= 0,
((F*AW + AP*A + AQ*AO) mod 2) #= 0, ((J*D + R*AR + W*AL) mod 2) #= 0,
((J*E + S*AK + AD*AS) mod 2) #= 0, ((J*F + T*AD + AK*X) mod 2) #= 0,
((J*G + U*W + AR*AE) mod 2) #= 0, ((K*AN + Z*B + AD*AT) mod 2) #= 0,
((K*AO + AA*I + AK*Y) mod 2) #= 0, ((K*AP + AB*P + AR*AF) mod 2) #= 0,
((L*AV + AH*P + AK*Z) mod 2) #= 0, ((L*AW + AI*I + AR*AG) mod 2) #= 0,
((M*AB + AP*B + AR*AH) mod 2) #= 0, ((R*AU + Z*C + AE*AM) mod 2) #= 0,
((R*AV + AA*J + AL*AF) mod 2) #= 0, ((R*AW + AB*Q + AS*Y) mod 2) #= 0,
((S*AO + AH*Q + AL*AG) mod 2) #= 0, ((S*AP + AI*J + AS*Z) mod 2) #= 0,
((T*AI + AP*C + AS*AA) mod 2) #= 0, ((Z*F + AH*R + AM*L) mod 2) #= 0,
((Z*G + AI*K + AT*S) mod 2) #= 0, ((AA*N + AP*D + AT*T) mod 2) #= 0,
((AH*U + AP*E + AU*M) mod 2) #= 0,

findall(Vars,label(Vars),Bag), length(Bag,NrSols), writeln(NrSols).
\end{lstlisting}
\end{center}
\caption{SWI-Prolog code for dimension 7 thin tables\label{code7}}
\end{figure}
		
	\subsection{Check for simplicity}\label{sec:simpcheck}
	The Meataxe is a library of routines implemented in GAP and Magma whose principal functionality is in testing for the irreducibility of modules for matrix algebras, based on Norton's Irreducibility Test. For a Lie algebra $L$, the adjoint representation $\ad:L\to \gl(L)$ is reducible if and only if there is a proper subspace $I$ in $L$ which is stable under $\ad(x)$ for all $x\in L$, i.e. if and only if $I$ has a proper ideal. (See \S\ref{homs} for background.) Given any thin table $T$ provided by Prolog, we constructed in GAP the $(2^{n}-1)\times (2^n-1)$ adjoint matrices of the $2^{n}-1$ basis elements, via $\ad(e_\alpha)e_\beta=T_{\alpha,\beta} e_{\alpha+\beta}$ for each $\alpha,\beta\in\Phi$. The Meataxe then detects simplicity or otherwise.
	
	The code for converting a thin table into a Lie algebra in GAP and checking its simplicity with the MeatAxe is reproduced in Appendix \ref{app:meataxe}.
	
	\subsection{Invariants: 2-nilpotent and toral elements}\label{sec:2nilptor}
	Let $L$ be a Lie algebra over a field $k$ of characteristic $p$. Then an element $x$ of $L$ is \emph{$p$-nilpotent} if $\ad(x)^p = 0.$ The $p$-nilpotent elements $N_L$ of a Lie algebra $L$ form an affine subscheme of $L$ which can be explicitly described. Let $\{e_1\dots, e_r\}$ be a basis of $L$. The adjoint matrix of a general element $x=\sum\lambda_ie_i$ is $X:=\sum \lambda_iE_i$, where $E_i$ is the adjoint matrix of $e_i$. Treating the $\lambda_i$ as variables gives $X$ the status of a $n\times n$ matrix with entries in  the polynomial algebra $R:=k[\lambda_1,\dots,\lambda_r]$ over $k$. The scheme $N_L$ is defined by the ideal $I$ generated by the entries $\{(X^p)_{ij}\mid 1\leq i,j\leq r\}$ of $X^p$. A $k$-point $x\in N_L(k)$ is by definition a $k$-algebra map $R\to k$ which factors through $I$; in a more down-to-earth description, it is just a choice of $(x_1,\dots,x_r)\in k^r$ such that evaluating $\lambda_k$ at $x_k$ in each $(X^p)_{ij}$ returns zero. If $k$ is finite, then the number of $k$-points $|N_L(k)|$ is an invariant of $L$, in the sense that if $L\cong L'$ then $|N_L(k)|=|N_{L'}(k)|$.
	
	Now assume $k=\GF(2)$ and $L$ is a thin Lie algebra on the basis $\{e_\alpha\mid\alpha\in\Phi\}$. Then we wish to count $|N_L(k)|$, i.e. to count the points of the variety $N_L$ over the finite field $k$. Counting points on varieties over finite fields is a very classical problem and there are various systems available for doing this calculation; but the state of the art methods for available to us in Magma, GAP or Singular (Sage) were unable to compute the number of $2$-nilpotent elements in any of the $31$-dimensional thin Lie algebras within 4 hours (which is when we killed the programs).
		
	Prolog, by contrast, was able to count $2$-nilpotent elements in under a second. 
	
	Let us provide a toy example to demonstrate this process using $\mathfrak{fsl}_2$. Recall its thin table from Example \ref{fsl2}.	One calculates that the basis elements have adjoint matrices
	$$
	E_{\alpha} = \begin{pmatrix}
	0 & 0 & 0\\
	0 & 0 & 1\\
	0 & 1 & 0
	\end{pmatrix},\:
	E_{\beta} = \begin{pmatrix}
	0 & 0 & 1\\
	0 & 0 & 0\\
	1 & 0 & 0
	\end{pmatrix},\:
	E_{\alpha+\beta} = \begin{pmatrix}
	0 & 1 & 0\\
	1 & 0 & 0\\
	0 & 0 & 0
	\end{pmatrix}.$$
	Thus a general element $x\in \mathfrak{fsl}_2$ is 
	
	$$ x = \begin{pmatrix}
	0 & c & b\\
	c & 0 & a\\
	b & a & 0
	\end{pmatrix},$$
	with $a,b,c\in \GF(2).$ Now,
	$$ x^2 = \begin{pmatrix}
	b^2+c^2 & ab & ac\\
	ab & a^2+c^2 & bc\\
	ac & bc & a^2+b^2
	\end{pmatrix}.$$
	Therefore setting $x^2 = 0$ gives the following 6 equations
	$$ab = ac = bc = a^2+b^2 = a^2+c^2 = b^2+c^2 = 0.$$
	Translated into Prolog code, this is:
	\begin{lstlisting}
	?- [A,B,C] ins 0..1, 
	A*B #= 0,
	A*C #= 0,
	B*C #= 0,
	((A^2+B^2) mod 2) #= 0,
	((A^2+C^2) mod 2) #= 0,
	((B^2+C^2) mod 2) #= 0,
	label([A,B,C]).
	\end{lstlisting}
	(In this case, the only solution is the trivial solution $a = b = c = 0.$)

	\subsection{Creating the search: full version}\label{sec:CushCode}

The full code can be found in the appendices. The program consists of four main parts, coordinated through a master predicate {\tt reduced\_thin\_search}:

\begin{figure}[h]\begin{lstlisting}	
reduced_thin_search( N, ThinTables ) :-
	findall( Rows, (thin_search( Vs, N, Rows), labeling( [], Vs) ), FirstTables ),
	remove_non_simple_tables( FirstTables, SimpleTables ),
	lex_reduce_tables( N, SimpleTables, ReducedTables ),
	perform_toral_switchings( N, ReducedTables, ThinTables ).
\end{lstlisting}\caption{The {\tt reduced\_thin\_search} master predicate}\end{figure}

An example query of this predicate is {\tt reduced\_thin\_search(4, ThinTables), maplist( writeln, ThinTables).} This query will search for thin tables that represent Lie algberas in dimension $15$ and print out the results.

This predicate holds true when $N$ is an integer and {\tt ThinTables} is a list of thin tables representing each simple thin Lie algebra of dimension $2^N - 1$ at least once. The first subsidiary predicate {\tt findall} does an initial pass to generate a large set of tables {\tt FirstTables}, most of which are discarded since they fail the subsequent predicates. Even so, {\tt ThinTables} may still contain pairs of tables which describe isomorphic algebras. Those are dealt with by a final post-process.

\subsubsection{  The initial thin search } 
\begin{figure}[h]\begin{lstlisting}	
thin_search( Vs, N, M, Rows) :-
	length(Rows, M),
	maplist(same_length(Rows),Rows),
	append(Rows, Vs), Vs ins 0..1,
	numlist(M, Indices),
	transpose(Rows, Rows),
	maplist( set_value_to_zero, Indices, Rows),
	stop_certain_ideals(Rows,Indices),
	act_faithfully(Rows,Indices),
	jacobi_identity_full( Indices, Rows),
	break_gl2_symmetries( Vs, Rows, N ).
\end{lstlisting}\caption{The initial {\tt thin\_search} predicate}\end{figure}
The predicate {\tt thin\_search( Vs, N, M, Rows)} is true when $M=2^N-1$ and {\tt Rows} is a $M\times M$ thin table satisfying some constraints necessary for simplicity and others which break symmetry. The variable {\tt Vs} is a flattened list of the variables in {\tt Rows} which is needed for the purpose of labelling under CLP(FD).

Specifically, the predicates {\tt stop\_certain\_ideals} and {\tt act\_faithfully} hold when the matrix {\tt Rows} satisfies the constraints derived from Lemma \ref{l0l1}; these predicates hold for thin tables describing simple Lie algebras, and do not hold for most thin tables. 

Let us explain the symmetry breaking predicate {\tt break\_gl2\_symmetries} by way of example. Consider the following two thin tables of dimension $7$. 

	\begin{center}\footnotesize
	\begin{tabular}{ c|c|c|c|c|c|c|c| } 
		& ${001}$ & ${010}$ & ${011}$ & ${100}$ & ${101}$ & ${110}$ & ${111}$ 
		\\ 
		\hline
		${001}$ & $0$ & $0$ & $0$ & $1$ & $1$ & $1$ & $1$ 
		\\ 
		\hline
		${010}$ & $0$ & $0$ & $1$ & $0$ & $0$ & $1$ & $1$ 
		\\ 
		\hline
		${011}$ & $0$ & $1$ & $0$ & $1$ & $1$ & $0$ & $0$  
		\\ 
		\hline
		${100}$ & $1$ & $0$ & $1$ & $0$ & $1$ & $1$ & $1$  
		\\ 
		\hline
		${101}$ & $1$ & $0$ & $1$ & $1$ & $0$ & $1$ & $1$  
		\\ 
		\hline
		${110}$ & $1$ & $1$ & $0$ & $1$ & $1$ & $0$ & $1$ 
		\\ 
		\hline
		${111}$ & $1$ & $1$ & $0$ & $1$ & $1$ & $1$ & $0$ 
		\\ 
		\hline
	\end{tabular}
	\hfill
	\begin{tabular}{ c|c|c|c|c|c|c|c| } 
		& ${001}$ & ${010}$ & ${011}$ & ${100}$ & ${101}$ & ${110}$ & ${111}$ 
		\\ 
		\hline
		${001}$ & $0$ & $0$ & $1$ & $0$ & $1$ & $0$ & $1$ 
		\\ 
		\hline
		${010}$ & $0$ & $0$ & $1$ & $1$ & $1$ & $1$ & $1$ 
		\\ 
		\hline
		${011}$ & $1$ & $1$ & $0$ & $1$ & $0$ & $1$ & $1$  
		\\ 
		\hline
		${100}$ & $0$ & $1$ & $1$ & $0$ & $0$ & $1$ & $1$  
		\\ 
		\hline
		${101}$ & $1$ & $1$ & $0$ & $0$ & $0$ & $1$ & $0$  
		\\ 
		\hline
		${110}$ & $0$ & $1$ & $1$ & $1$ & $1$ & $0$ & $1$ 
		\\ 
		\hline
		${111}$ & $1$ & $1$ & $1$ & $1$ & $0$ & $1$ & $0$ 
		\\ 
		\hline
	\end{tabular}
\end{center}

These two tables represent isomorphic Lie algebras; this can be seen by transposing the simple roots $001$ and $010$ and extending this permutation linearly to get another thin table (which means also tranposing $101$ and $011$). Our predicate discards a large subset of isomorphic algebras coming from such permutations of the roots. Now, any permutation of the roots preserving is determined on the simple roots, which are a standard basis of $\GF(2)^N$. Let $M$ be element of $\GL_N(2)$, i.e.~an $N\times N$ invertible matrix over $\GF(2)$. (In our example $N=3$.) We may turn $M$ into a symmetry breaking constaint by replacing the entry $T(\alpha,\beta)$ by $T(M\alpha,M\beta)$, giving a thin table $T_M$ which describes an isomorphic Lie algebra. We now impose the constraint that $T$ is lexicographically lower than $T_M$ when flattened as lists. 

For $N$ larger than $3$, it is not practical to use all elements of $\GL_N(2)$ in the initial search since the group becomes very large with $N$. Instead, we use the set of all five non-identity elements of each of the $N\choose 2$ subgroups $\GL_2(2)$ coming from each pair of simple roots. 

Querying Prolog {\tt thin\_search} with $N = 3$ results in the following $4$ tables representing the same $7$ dimensional simple thin Lie algebra which we denote by $T1, T2, T3$ and $T4.$

	\begin{center}\footnotesize
	\begin{tabular}{ c|c|c|c|c|c|c|c| } 
		& ${001}$ & ${010}$ & ${011}$ & ${100}$ & ${101}$ & ${110}$ & ${111}$ 
		\\ 
		\hline
		${001}$ & $0$ & $0$ & $0$ & $1$ & $1$ & $1$ & $1$ 
		\\ 
		\hline
		${010}$ & $0$ & $0$ & $1$ & $0$ & $0$ & $1$ & $1$ 
		\\ 
		\hline
		${011}$ & $0$ & $1$ & $0$ & $1$ & $1$ & $0$ & $0$  
		\\ 
		\hline
		${100}$ & $1$ & $0$ & $1$ & $0$ & $1$ & $1$ & $1$  
		\\ 
		\hline
		${101}$ & $1$ & $0$ & $1$ & $1$ & $0$ & $1$ & $1$  
		\\ 
		\hline
		${110}$ & $1$ & $1$ & $0$ & $1$ & $1$ & $0$ & $1$ 
		\\ 
		\hline
		${111}$ & $1$ & $1$ & $0$ & $1$ & $1$ & $1$ & $0$ 
		\\ 
		\hline
	\end{tabular}
	\begin{tabular}{ c|c|c|c|c|c|c|c| } 
		& ${001}$ & ${010}$ & ${011}$ & ${100}$ & ${101}$ & ${110}$ & ${111}$ 
		\\ 
		\hline
		${001}$ & $0$ & $0$ & $0$ & $1$ & $1$ & $1$ & $1$ 
		\\ 
		\hline
		${010}$ & $0$ & $0$ & $1$ & $0$ & $0$ & $1$ & $1$ 
		\\ 
		\hline
		${011}$ & $0$ & $1$ & $0$ & $1$ & $1$ & $1$ & $1$  
		\\ 
		\hline
		${100}$ & $1$ & $0$ & $1$ & $0$ & $1$ & $1$ & $1$  
		\\ 
		\hline
		${101}$ & $1$ & $0$ & $1$ & $1$ & $0$ & $1$ & $1$  
		\\ 
		\hline
		${110}$ & $1$ & $1$ & $1$ & $1$ & $1$ & $0$ & $0$ 
		\\ 
		\hline
		${111}$ & $1$ & $1$ & $1$ & $1$ & $1$ & $0$ & $0$ 
		\\ 
		\hline
	\end{tabular}
\end{center}
	\begin{center}\footnotesize
	\begin{tabular}{ c|c|c|c|c|c|c|c| } 
		& ${001}$ & ${010}$ & ${011}$ & ${100}$ & ${101}$ & ${110}$ & ${111}$ 
		\\ 
		\hline
		${001}$ & $0$ & $0$ & $1$ & $1$ & $1$ & $1$ & $1$ 
		\\ 
		\hline
		${010}$ & $0$ & $0$ & $1$ & $1$ & $1$ & $1$ & $1$ 
		\\ 
		\hline
		${011}$ & $1$ & $1$ & $0$ & $0$ & $1$ & $1$ & $0$  
		\\ 
		\hline
		${100}$ & $1$ & $1$ & $0$ & $0$ & $0$ & $0$ & $1$  
		\\ 
		\hline
		${101}$ & $1$ & $1$ & $1$ & $0$ & $0$ & $1$ & $1$  
		\\ 
		\hline
		${110}$ & $1$ & $1$ & $1$ & $0$ & $1$ & $0$ & $1$ 
		\\ 
		\hline
		
		${111}$ & $1$ & $1$ & $0$ & $1$ & $1$ & $1$ & $0$ 
		\\ 
		\hline
	\end{tabular}
	\begin{tabular}{ c|c|c|c|c|c|c|c| } 
		& ${001}$ & ${010}$ & ${011}$ & ${100}$ & ${101}$ & ${110}$ & ${111}$ 
		\\ 
		\hline
		${001}$ & $0$ & $0$ & $1$ & $1$ & $1$ & $1$ & $1$ 
		\\ 
		\hline
		${010}$ & $0$ & $0$ & $1$ & $1$ & $1$ & $1$ & $1$ 
		\\ 
		\hline
		${011}$ & $1$ & $1$ & $0$ & $0$ & $1$ & $1$ & $0$  
		\\ 
		\hline
		${100}$ & $1$ & $1$ & $0$ & $0$ & $1$ & $1$ & $1$  
		\\ 
		\hline
		${101}$ & $1$ & $1$ & $1$ & $1$ & $0$ & $1$ & $0$  
		\\ 
		\hline
		${110}$ & $1$ & $1$ & $1$ & $1$ & $1$ & $0$ & $0$ 
		\\ 
		\hline
		
		${111}$ & $1$ & $1$ & $0$ & $1$ & $0$ & $0$ & $0$ 
		\\ 
		\hline
	\end{tabular}
\end{center}

\subsection{Post search simplicity check}
Let $L$ be a thin Lie algebra with roots $\Lambda.$ Let $\Lambda_0\subset\Lambda$ such that for all $\alpha\in\Lambda_0,$
$$[ e_\alpha, e_{ \alpha + \beta} ] = e_{\beta} {\rm \: for \: all\: }  \beta\in \Lambda_0, \beta\neq\alpha,$$
$$[ e_\alpha, e_{ \alpha + \beta} ] = 0 {\rm \: for \: all\: }  \beta\notin \Lambda_0.$$

Then $I = \{ e_{\alpha} | \alpha\in \Lambda_0 \}$ is an ideal in $L$ and thus $L$ is not simple. 
%

The predicate {\tt remove\_non\_simple\_tables} removes all tables with such an ideal.

\subsection{Lexicographic reduction}\label{sec:lex} As mentioned above it is not plausible to implement all possible symmetry breaking constraints, especially in dimension 31. The predicate {\tt lex\_reduce\_tables} removes all tables $T$ which are permutation equivalent to a smaller table in lexicographical order. 

In the $7$ dimensional case one sees that $T1, T2$ and $T3$ are permutation equivalent, with $T1$ being the smallest in lexicographical order. This leaves just $T1$ and $T2$ to consider.
%

\subsection{Toral Switching}\label{sec:tor}  The final stage of the Prolog code performs toral switching to remove equivalent tables. 
\begin{lstlisting}	
perform_toral_switchings( N, Tables, SwitchedTables ) :-
	M is 2^N-1,
	K is 2^(N-1),
	numlist( M, Indices ),
	get_roots( N,Roots ),
	make_powers( N, Powers ), 
	maplist(locate_basis_nilps( Indices, K ), Tables, NilpIndicesList ),
	make_base_torus( BaseTorus, N, Indices ),
	findall( [A,B], ( member(A, Indices), member(B, Indices), A #< B ), Pairs ), 
	maplist( make_torus_switches( Indices, Pairs, Roots, BaseTorus ), Tables, NilpIndicesList, SwitchedTablesList ),
	once(canonical_order_tables_list( N, Roots, Powers, Tables, SwitchedTablesList, OrderedTablesList )), 
	once(make_adjacency_mat( Tables, OrderedTablesList, AdjacencyMat )),
	adjacency_mat_to_ugraph( AdjacencyMat, SwitchingGraph ),
	conn_comps( SwitchingGraph, Comps ),
	maplist( get_min_from_comp( Tables ), Comps, SwitchedTables ).
\end{lstlisting}
Let $T$ be a thin table with basis $(e_\alpha)_{\alpha\in\Phi}$ and simple roots $(\alpha_i).$ Throughout we use the torus $\mathscr{T}=\{t_i\mid 1\leq i\leq n\in \mathscr{L}\}$ where $[t_i, e_{\alpha_i}] = e_{\alpha_i} $ and $[t_i, e_{\alpha_j}] = 0$ for $i\neq j.$  The predicate {\tt make\_base\_torus} is true when {\tt Base\_Torus} is a variable representing the elements of this torus.

Suppose $\beta\in\Phi$ such that $e_\beta$ is nilpotent. Then by the theory of toral switching, the subspace $\langle t_i + (\alpha_i, \beta) e_\beta\mid 1\leq i\leq n\rangle$ is a torus of maximal rank in $\mathscr L$. Let $T(\beta)$ be the thin table obtained by decomposing $L$ with respect to this switched torus. 

Given a thin table the predicate {\tt make\_torus\_switches} holds whenever the variable $\tt SwitchedTablesList$ contains all such toral switches from the base torus.

For example, the dimension $7$ table $T1$ has nilpotent elements $e_{010}$ and $e_{011}.$ Toral switching with respect to these elements produces the following two tables.

	\begin{center}\footnotesize
	\begin{tabular}{ c|c|c|c|c|c|c|c| } 
		& ${001}$ & ${010}$ & ${011}$ & ${100}$ & ${101}$ & ${110}$ & ${111}$ 
		\\ 
		\hline
		${001}$ & $0$ & $0$ & $1$ & $1$ & $1$ & $1$ & $1$ 
		\\ 
		\hline
		${010}$ & $0$ & $0$ & $1$ & $0$ & $0$ & $1$ & $1$ 
		\\ 
		\hline
		${011}$ & $1$ & $1$ & $0$ & $1$ & $1$ & $1$ & $0$  
		\\ 
		\hline
		${100}$ & $1$ & $0$ & $1$ & $0$ & $1$ & $0$ & $1$  
		\\ 
		\hline
		${101}$ & $1$ & $0$ & $1$ & $1$ & $0$ & $1$ & $1$  
		\\ 
		\hline
		${110}$ & $1$ & $1$ & $1$ & $0$ & $1$ & $0$ & $1$ 
		\\ 
		\hline
		${111}$ & $1$ & $1$ & $0$ & $1$ & $1$ & $1$ & $0$ 
		\\ 
		\hline
	\end{tabular}
	\begin{tabular}{ c|c|c|c|c|c|c|c| } 
		& ${001}$ & ${010}$ & ${011}$ & ${100}$ & ${101}$ & ${110}$ & ${111}$ 
		\\ 
		\hline
		${001}$ & $0$ & $1$ & $1$ & $1$ & $0$ & $1$ & $1$ 
		\\ 
		\hline
		${010}$ & $1$ & $0$ & $1$ & $0$ & $1$ & $1$ & $1$ 
		\\ 
		\hline
		${011}$ & $1$ & $1$ & $0$ & $1$ & $1$ & $0$ & $1$  
		\\ 
		\hline
		${100}$ & $1$ & $0$ & $1$ & $0$ & $1$ & $0$ & $1$  
		\\ 
		\hline
		${101}$ & $0$ & $1$ & $1$ & $1$ & $0$ & $1$ & $1$  
		\\ 
		\hline
		${110}$ & $1$ & $1$ & $0$ & $0$ & $1$ & $0$ & $0$ 
		\\ 
		\hline
		${111}$ & $1$ & $1$ & $1$ & $1$ & $1$ & $0$ & $0$ 
		\\ 
		\hline
	\end{tabular}
\end{center}
	
For each switched table $T$ generated, the predicate {\tt canonical\_order\_tables\_list} calculates the lexicographically smallest under permutation equivalence, say $LM(T).$ Every table generated this way must be an element of {\tt ReducedTables}.

Finally we construct a graph with vertex set represented by the tables {\tt ReducedTables} in the search. And edge is drawn between tables $T_1$ and $T_2$ if there exists a switching of $T_1$ to a new table $T_1'$ such that $LM(T_1')=LM(T_2)$.  For each connected component in this graph we take the table which is the smallest with respect to lexicographical order. That gives our final list of tables denoted by {\tt ThinTables}.

In the dimension $7$ search we see that $T1(010)$ is equivalent to $T2$. Thus in dimension $7$ the search produces precisely one thin table.
%
\subsection{Some final post-processing}\label{sec:finalfinal}
The result of the full search initiated by the predicate {\tt reduced\_thin\_search} is a list of $9$ or $10$ tables in dimensions $15$ or $31$ respectively.

In case $\dim L=15$, each pair of Lie algebras arising from each table is seen to be non-isomorphic by comparison of their $2$-nilpotent and toral elements, with two exceptions. In the exceptional cases each pair of Lie algebras was found to be isomorphic.

In case $\dim L=31$, each pair of Lie algebras is seen to be non-isomorphic by comparison of their $2$-nilpotent and toral elements, again with two exceptions. In one exceptional case (I), we found an invariant distinguishing the algebras. In case (II) the algebras turned out to be isomorphic.

Since the calculations needed to establish the above are rather technical and do not appear to have wider applicability, we describe them rather tersely below.

It is typically not a straightforward task to find an isomorphism between two Lie algebras by brute force. The total number of linear maps between two $31$-dimensional Lie algebras over $\GF(2)$ is $2^{31^2}$ and it is not obvious how to search quickly for one of the very few which turns out to respect the Lie bracket. (The general method given in \cite{Eic10}, for example, cannot be used, as it requires iterating over all elements of $L$.) However, we were fortunate to notice that in each of the cases we considered, there was a unique non-zero $2$-nipotent element of minimal rank. In case $\dim L=15$, this fact was enough to run a constraint-based search in Prolog to locate an isomorphism.

In case $\dim L=31$ we needed to do rather more work. To rule out an isomorphism between the Lie algebras $P_1$ and $P_2$ in case (I), we took the respective sets of $2$-nilpotent elements $N_1$ and $N_2$, say,  and found in each a unique non-zero $2$-nilpotent of minimal rank, $n_1$ and $n_2$, say. Then we calculated the invariant $|(\ad(n_i)\cdot N_i) \cap N_i|$ for each algebra. It turned out that there were 2048 elements $n\in N_1$ such that $\ad(n_1) n\not\in N_1$ but 4096 elements $n\in N_2$ such that $\ad(n_2) n\not\in N_2$. Therefore $N_1\not\cong N_2$.

To find an isomorphism between the two remaining algebras, say $\TW$ and $\WT$, we aimed to find bases of each algebra whose associated structure constant tables matched. To that end, we had Prolog output a list of the toral elements of each algebra and used GAP to find the subset of the $2$-nilpotent elements which constituted the so-called \emph{sandwich elements}. The latter elements are those $2$-nilpotent elements $n$ satisfying $[[n,x],[n,y]]=0$ for all $x,y\in L$. One checks that the condition for a $2$-nilpotent element $n$ to be a sandwich is equivalent to saying that the linear map \[(1+\ad n): L\to L\] induces an automorphism of $L$. Exponentiating the sandwich elements in this way generated a group of automorphisms which we used to partition the toral elements into orbits. Looking at those orbits which lay outside the centralisers of all nilpotent elements reduced the set of orbits significantly. Some experimentation established that one could generate the whole of $\TW$ and $\WT$ using just one of these orbits $\O$ together with a single element $x$ from another orbit $\O'$. In fact, the set of pairs $(\O,x)$ such that $\langle \O,x\rangle=L$ could be refined further by imposing the condition that the set $\{t\in \O \mid [t,x]^{[2]}=0\}$ had just $2$ elements, $t_1$ and $t_2$ say, necessarily linearly independent. We also observed that $\dim k\O=6$. All bases of $k\O$ which contained $t_1$ and $t_2$ were collected. Multiplying together all pairs of elements in each of these bases and taking their ranks gave a list of signatures. Some signatures occurred rather rarely. We used one basis with a rare signature together with the extra element $x$ to generate a basis in  $\TW$, keeping track of the $31$ words $\W$ in the generators used to produce each basis element. We used this basis to get a structure constant table for $\TW$. Turning to $\WT$, we considered the triples $(\O,x,B)$ where $t_1,t_2\in B\subseteq \O$ and $B$ is a basis of $k\O$ generating $L$ with $x$ and having rare signature. Taking each ordering of $B$ in turn, we calculated appropriate words in $B$ and $x$ corresponding to those in $\W$. For the collection of words that gave a basis, we calculated a structure constant table for $\WT$. Eventually we hit upon a combination which gave the same structure constant table that we had calculated for $\TW$ and so concluded they were isomorphic.

\subsection{ Running the search for dimensions 7, 15, 31 }
The following list shows the number of thin tables satisfying successive predicates in dimensions $7, 15$ and $31$. The table {\tt FinalTables} denotes the exact number of isomorphism classes of thin Lie algebras after the final post-processing above.
	\begin{center}
	\begin{tabular}{ c|c|c|c| } 
		& $7$ & $15$ & $31$ 
		\\ 
		\hline
		{\tt FirstTables} & $4$ & $127$ & $5532$ 
		\\ 
		\hline
		{\tt SimpleTables} & $4$ & $99$ & $5532$ 
		\\ 
		\hline
		{\tt ReducedTables} & $2$ & $22$ & $72$ 
		\\ 
		\hline
		{\tt ThinTables} & $1$ & $9$ & $10$ 
		\\ 
		\hline
		{\tt FinalTables} & $1$ & $7$ & $9$ 
		\\ 
		\hline
	\end{tabular}
\end{center}
Text files containing the contents of {\tt FinalTables}  in dimensions $3, 7, 15$ and $31$ can be found at \url{github.com/cushydom88/simple-lie-algebras} .

\section{The results}

	\begin{thm} There are six isomorphism classes of simple thin Lie algebras of dimension $15$, namely ${\rm Kap}_2 (4), P(2;\underline{2}), P(3;[2,1,1]), {\rm Skry}, {\rm Eick}_{15,1}$, ${\rm Eick}_{15,2}$ and a previously unknown algebra we denote $N_1(4)$.
		
		There are nine isomorphism classes of simple thin Lie algebras of dimension $31$; these are labelled $P(3;[2,2,1])$, $P(4;[2,1,1,1])$ together with seven previously unknown algebras we call $\mathrm{LD,\ TW,\ LS,\ BW,\ GN,\ DL\text{ and }CB}$.\end{thm}

	\begin{rmk}
		(i) The algebra $N_1(4)$ has been discovered independently, in \cite{Em22}.
		
		(ii) The Lie algebras with labels $P(3;[2,2,1])$ and $P(4;[2,1,1,1])$ are previously known, and simple of dimension $31$. We have checked that they are thin. By contrast, the other known algebras of dimension $31$ are not thin, namely $Q(5, \underline{1}), W(1,\underline{5})', P(2;[4,1]), P(2;[3,2])$ and $P(3;[3,1,1]).$ We checked this in GAP using Lemma \ref{thinvsthin}. 

		(iii) If only the zero diagonal and symmetry constraints are used there are $2^{465}\approx 9.53\times 10^{139}$ possible tables in which dimension $31$. It is reasonable to say that searching for tables generating simple Lie algebras would have been hopeless using standard computational methods.
	\end{rmk}

	\section{New infinite families of simple Lie algebras}\label{sec:new}
	We end by introducing two new infinite families of simple (thin) Lie algebras. These were discovered by hand, extrapolating from the $15$-dimensional algebras $N_1(4)$ and $\mathrm{Eick}_{15,2}$ and some $63$-dimensional analogues found by Prolog. We will define these algebras over an arbitrary field $k$ of characteristic $2$.
	
	\subsection{Definitions}
	
	Fix $n\in\mathbb{N}.$ For $i\in {1,\ldots, n}$ we let $\alpha_{i}\in \GF(2)^n$ be the standard basis. Then according the usual inner product, $(\alpha,\alpha_i)$ is the coefficient of $\alpha$ according to that basis. Thus $(\alpha,\beta)=\alpha\cdot\beta=\sum_{i}[(\alpha,\alpha_i)(\beta,\alpha_i)]$.  In this case, the $\ell^{1}$ norm $|\alpha|$ equals the Euclidean norm $||\alpha||=(\alpha,\alpha)^{1/2}$ and both identify with the parity of the sum of the coefficients of $\alpha$ according to the $\alpha_i$. Namely,
	$$\left|\left( \sum_{i=1}^{n}x_i \alpha_i \right)\right| = \sum_{i=1}^{n}x_i \in \GF(2). $$
	
	Let $\Phi:=\GF(2)^{n} \setminus \underline{0}$ and let $N_{1}(n, k)$ and $N_2(n,k)$ be $(2^n-1)$-dimensional $k$-vector spaces on the basis $\left\{ e_{\alpha} : \alpha\in \Phi\right\}$. The Lie bracket is determined by $[e_\alpha,e_\beta]$ for $\alpha\neq\beta$ as follows:
	
	\underline{For $N_1(n,k)$:}
	\begin{enumerate}
		\item 
		Suppose such that $|\alpha| = |\beta| = 1$. Then we set
		$$[e_{\alpha}, e_{\beta}] = e_{\alpha + \beta}.$$
		\item
		Suppose that $|\alpha|=0$ or $|\beta|=0$. Then we set
		
		$$
		[e_{\alpha}, e_{\beta}] =\begin{cases}
		e_{\alpha + \beta}, & {\rm if }\: (\alpha, \beta) = 1,\\
		0, & \text{otherwise.}
		\end{cases}
		$$
	\end{enumerate}
		
		\underline{For $N_2(n,k)$:}
	\begin{enumerate}
		\item 
		Suppose that $(\alpha,\alpha_n) = 0$ and $|\alpha| = |\beta| =1$. Then we set
		$$[e_{\alpha}, e_{\beta}] = e_{\alpha + \beta}.$$
		\item
		Suppose that $(\alpha,\alpha_n) = 0$ and $|\alpha| = 0$. Then we set
		$$
		[e_{\alpha}, e_{\beta}] =\begin{cases}
		e_{\alpha + \beta}, & {\rm if }\: (\alpha,\beta) = 1,\\
		0, & \text{otherwise}.
		\end{cases}
		$$
		\item 
		Suppose that $(\alpha,\alpha_n) = (\beta,\alpha_n) = 1$. Then we set
		$$
		[e_{\alpha}, e_{\beta}] =\begin{cases}
		0, & {\rm if }\: |\alpha| =  |\beta| = 1, \\
		e_{\alpha + \beta}, & \text{otherwise}.
		\end{cases}
		$$
	\end{enumerate}
		Note that $N_1(n-1,k)$ is a subalgebra of $N_2(n,k)$ via the subset $\{\alpha\in\Phi\mid (\alpha,\alpha_n)=0\}$.

		It is simple, though tedious, to check that $[\underline{\hspace{.25cm}},\underline{\hspace{.25cm}}]$ is a valid Lie bracket according to the various cases. We do this explicitly for one case for $N_1(n,k)$ and omit the remainder for brevity.

	Suppose that $\alpha, \beta, \gamma \in \GF(2)^{n} \setminus \underline{0} $ with $|\alpha| = |\beta| = |\gamma| =1$. 
	Then
	$$[e_{\alpha}, [e_{\beta}, e_{\gamma}] ] = [e_{\alpha}, e_{\beta+\gamma}] = (\alpha, (\beta +\gamma )) e_{\alpha + \beta + \gamma} = ((\alpha,\beta)+(\alpha,\gamma)) e_{\alpha + \beta + \gamma}.$$
	Similarly
	$$[e_{\beta}, [e_{\gamma}, e_{\alpha}]] = ((\alpha,\beta)+(\beta,\gamma)) e_{\alpha + \beta + \gamma}\text{, and}$$
	$$[e_{\gamma}, [e_{\beta}, e_{\alpha}]] = ((\alpha,\gamma)+(\beta,\gamma))e_{\alpha + \beta + \gamma} .$$
	Thus
	$$[e_{\alpha}, [e_{\beta}, e_{\gamma}] ] + [e_{\beta}, [e_{\gamma}, e_{\alpha}]] + [e_{\gamma}, [e_{\beta}, e_{\alpha}]] = (2(\alpha,\beta) +2(\beta,\gamma) + 2(\alpha,\gamma))e_{\alpha + \beta + \gamma} = 0,$$
	as required.

	We write $N_1(n)$ and $N_2(n)$ for $N_1(n,\GF(2))$ and $N_2(n,\GF(2))$ respectively.
	
	\subsection{The simplicity of the new infinite families}
	In everything that follows,  $L$ will be one of the thin algebras $N_1(n,k)$ or $N_2(n,k)$ with $k$-basis $\{e_\alpha:\alpha\in\Phi\}$ and Lie bracket as defined in the last section.
	
	Direct calculation from the definitions above show $\ad(e_{\alpha_i})^2$ is a toral element $t_i\in\Der(L)$ which centralises $e_{\alpha_i}$ but has weight $1$ on each $e_{\alpha_j}$ with $i\neq j\in\{1,\dots, n\}$; i.e. $t_i(e_{\alpha_j})=\delta_{ij}e_{\alpha_j}$ for $\delta_{ij}$ the Kronecker delta. 
	
	First assume $n=2m$ is even. Then the element $\hat t_i:=\sum_{j\neq i} t_i$ thus has weight $1$ on $e_{\alpha_i}$ and centralises $e_{\alpha_j}$ for $j\neq i$.  Therefore $\{\hat t_i\mid 1\leq i\leq n\}$ is a basis for a torus $T\subseteq\mathscr{L}$. We now show:
		
	\begin{lemma}\label{lemma:torus} For $n=2m$, the algebra $L$ has a thin decomposition for which $\hat t_i(e_\alpha)=(\alpha,\alpha_i)e_\alpha$.\end{lemma}
\begin{proof}
	We calculate the weights of $T$ on $L$. If $|\alpha|=1$ then from the definitions, \[t_i(e_\alpha)=\begin{cases}e_\alpha\text{ whenever }(\alpha,\alpha_i)=0,\\0\text{ otherwise.}\end{cases}\]
	and so
\[\hat t_i(e_\alpha)=\begin{cases}e_\alpha\text{ whenever }(\alpha,\alpha_i)=1,\\0\text{ otherwise.}\end{cases}\]
It follows that $\alpha$ is a root for $T$ such that $\alpha(\hat t_i)=\alpha_i$. A similar analysis gives the same conclusion in case $|\alpha|=0$. Thus $T$ acts on $L$ with one-dimensional root spaces, with $\hat t_i$ the coroots of $\alpha_i$.
	\end{proof}
	\begin{thm}\label{thm:simples} For $n=2m$, the algebra $L$ is absolutely simple.
	\end{thm}
	\begin{proof}Our proof will show the simplicity of $L$ with $k$ arbitrary of characteristic $2$. 
	
	Suppose $0\neq x\in I\triangleleft L$ is an ideal. Write $x=\sum_{\alpha\in\Phi}x_{\alpha}e_\alpha$ and pick $\beta$ such that $x_\beta = 1.$ Since $t_i=\ad(e_i)^2$, we have $t_i$ and $\hat t_i$ stabilise $I$ for each $1\leq i\leq n$. Now write $\beta=\sum \beta_i\alpha_i$ for $\beta_i=(\beta,\alpha_i)$. The proof of Lemma \ref{lemma:torus} in fact shows that the derivation $t_\beta=\prod_{\beta_i=1} (\hat t_i)\circ \prod_{\beta_i=0}(t_i)$ satisfies $t_\beta(x)=x_\beta e_\beta\in I$.
	
Pick $j$ such that $\beta_j = 1.$ Then either $\beta=\alpha_j$ or $\beta+\alpha_j\in\Phi$. In the latter case, $|\beta|+|\beta+\alpha_j|=1\in\GF(2)$ for any $\beta\in\Phi$ and so it follows that $[ e_\beta, e_{\beta + \alpha_j} ] = e_{\alpha_j}\in I.$ Similarly, if $\beta\neq \alpha_j$ is a root for which $\beta_j=1$ then $[ e_{\beta+\alpha_j},e_{\alpha_j}]=e_\beta$, hence $I=L$.\end{proof}
	


	\begin{thm}
	For $n=2m+1$ with $m\geq 2$, then $\DD(L)=[L,L]$ is absolutely simple of dimension $2^n-2$. 
	\end{thm}

	\begin{proof}
	Again, our proof will show the simplicity with $k$ taken to be arbitrary of characteristic $2$.
	
	Let $\hat\alpha = \sum \alpha_i.$ Direct calculation from the definition shows that $\DD(L) = \langle e_\alpha \mid \alpha\neq \hat\alpha \rangle.$ In particular $\DD(L)$ has dimension $2^n-2.$
	
	 Note $\hat\alpha(T)=0$. We have $N_1(2m,k)=\langle e_{\alpha}\mid(\alpha,\alpha_n)=0\rangle$ is a subalgebra of $\DD(L)$. Let $S=\langle \hat t_i\mid 1\leq i\leq 2m\rangle$ be the torus of rank $2m$ associated with this subalgebra as introduced above; thus \[\hat t_i=\sum_{\substack{1\leq j\leq 2m\\j\neq i}}t_i.\]
	
	One checks that the the kernel of $S$ on $L$ is just $\langle e_{\hat\alpha}\rangle$ and so every root space obtained by decomposing $\DD(L)$ with respect to $S$ is $2$-dimensional of the form $\langle e_\alpha, e_{\alpha+\hat\alpha}\rangle$. Without loss of generality, $(\alpha,\alpha_n)=0$, so $\hat t_i(\alpha)=(\alpha,\alpha_i)=t_i(\alpha+\hat\alpha)$. 

Now assume $0\neq x\in I\triangleleft \DD(L)$ with say $x_{\beta}=1$ for some $\beta\in\Phi$. Letting \[t_\beta:=\prod_{\substack{1\leq i\leq n-1,\\ (\beta,\alpha_i)=1}}\hat t_i\prod_{\substack{1\leq i\leq n-1,\\ (\beta,\alpha_i)=0}}t_i,\] we get $t_\beta(x)=x_\beta e_\beta+x_{\beta+\hat\alpha} e_{\beta+\hat\alpha}$.
	
We wish to exhibit a root vector in $I$ and we approach this in cases. 

Suppose $|\beta|=1$ with $\beta\neq\alpha_j$ for any $j$ and suppose $(\beta,\alpha_i)=1$ for some $1\leq i\leq n$. Then $[e_{\beta},e_{\alpha_j}]=e_{\beta+\alpha_j}$; meanwhile $(\beta+\hat\alpha,\alpha_i)=0$, $|\beta+\hat\alpha|=0$ and hence $[e_{\beta+\hat\alpha},e_{\alpha_j}]=0$. Thus $e_{\beta+\alpha_j}\in I$. If $\beta=\alpha_j$, then choose $i,k$ such that $i,j,k$ are pairwise distinct. Then $[e_\beta,e_{\alpha_i+\alpha_j+\alpha_k}]=e_{\alpha_i+\alpha_k}$ whereas $[e_{\beta+\hat\alpha},e_{\alpha_i+\alpha_j+\alpha_k}]=0$. In any case we get a single root vector in $I$. The argument when $|\beta|=0$ is similar.
	
	Thus we have shown $e_\beta \in I$ for some $\beta\in \Phi\setminus\{\hat\alpha\}.$ The rest of the proof follows as in Theorem \ref{thm:simples}.\end{proof}
	
	\subsection*{Acknowledgments} The authors are supported by the Leverhulme Trust Research Project Grant number RPG-2021-080.

\begin{appendices}
	
\section{Sicstus Prolog code}
\begin{lstlisting}
:- use_module(library(clpfd) ).
:- use_module(library(lists)).
:- use_module(library(samsort)). 
:- use_module(library(ordsets)).
:- use_module(library(ugraphs)).
:- use_module(library(between)).

% The following predicate reduced_thin_search is our main predicate for searching for simple thin Lie algebras. 
% The parameter N is inputted by the users. The output is ThinTables which is a list of NxN arrays. Each of which represents a thin Lie algebra of dimension 2^N - 1
% For N = 2,3,4,5 each member of ThinTables represents a simple Lie algebra. We confirm this using graph
% For N >= 4 there are different members of ThinTables that represent the same Lie algebra. See the main body of the paper for does_not_centralise

% reduced_thin_search consists of the following 4 main parts

% 1. INITIAL THIN SEARCH
% The predicate thin_search generates FirstTables which consists of a list of tables each representing a thin Lie algebra
% thin_search itself consists of three main componenets:
% a. Implement the Lie bracket as constraints so that each table produced represents a valid Lie algebra 
% b. Constraints which are necessary for simplicity of the represented Lie algebras via the predicates
%    stop_certain_ideals and act_faithfully. See the paper for the theoretical underpinning of these constraints.
% c. Symmetry breaking constraints via the predicate break_gl2_symmetries. These constraints removes some tables
%    which represent the same Lie algebra. Not all possible symmetries can be added as constraints due to times
%    and memory issues.  
% Note that the order of the predicates inside thin_search do not follow the order listed above.
% The order constraints are added impact the run time. We have optimised the ordering. 

% 2. POST SEARCH SIMPLICTY CHECK
% The predicate remove_non_simple_tables takes in the tables FirstTables and removes certain tables representing
% non simple Lie algebras. See the paper for a description of what ideals are checked for .
% The tables which pass this check are represented by SimpleTables.
% In theory this condition could have been converted to a constraint and added to thin_search.
% However doing this is significantly slower than performing them as a post process.

% 3. LEX REDUCE TABLES
% As mentioned in 1. the predicate break_gl2_symmetries does not implement all possible symmetries as constraints.
% We now implement the rest of these symmetries on the tables SimpleTables and output the results to ReducedTables.
% At this stage no two tables in ReducedTables can be transformed in to each other via permutations of roots. 

% 4. TORAL SWITCHING 
% We now remove further tables from ReducedTables representing the same Lie alegbra.
% For each table in ReducedTables we perform one toral switch for each nilpotent basis element in that table.
% A graph is then constructed with vertex set represented by the elements of ReducedTables 
% and an edge between two vertices if a toral switch transformed the tables represented by the two vertices in to one another.
% For each connected componenet of this graph we take one table from it and output these to ThinTables.

% In addition to the above sections we have the following two sections that collect shared predicates used throughout the code

% 5. SYMMETRY PREDICATES
% These are a collection of predicates used throughout the code that deal with permuting roots.

% 6. UTILITY PREDICATES
% These are a collection of generic predicates used throughout the code.

% An example prompt to run the main predicate is 
% ['ThinSymmetrySicstus.pl'].
% reduced_thin_search(4, ThinTables), maplist( writeln, ThinTables).

%%% 0. THE MAIN PREDICATE %%%
reduced_thin_search( N, ThinTables ) :-
	findall( Rows, (thin_search( Vs, N, Rows), labeling( [], Vs) ), FirstTables ),
	remove_non_simple_tables( FirstTables, SimpleTables ),
	lex_reduce_tables( N, SimpleTables, ReducedTables ),
	once(perform_toral_switchings( N, ReducedTables, ThinTables )).
%%%%%%

%%% 1. INITIAL THIN SEARCH %%%
thin_search( Vs, N, Rows) :-
	M is 2^N-1,
	length(Rows, M),
	maplist(same_length(Rows),Rows),
	append(Rows, Vs), 
	domain( Vs, 0, 1),
	transpose(Rows, Rows), % Lie Bracket constraint
	numlist(M, Indices),
	maplist( set_value_to_zero, Indices, Rows), % Lie Bracket constraint
	stop_certain_ideals(Rows,Indices), % Simplicity constraints
	act_faithfully(Rows,Indices), % Simplicity constraints
	jacobi_identity_full( Indices, Rows), % Lie Bracket constraint
	break_gl2_symmetries( Vs, Rows, N ). % Symmetry breaking constraints

% The Jacobi Identity 
jacobi_identity_full( Indices, Rows) :-
	findall( [A,B,C], ( append( [_,[A],_,[B],_,[C],_], Indices ),  D is xor(B, C), A #\= D ), Triples ),
	maplist( jacobi_identity( Rows ), Triples ).
	jacobi_identity( Rows, [I1,I2,I3] ) :-
	I4 is xor(I1, I2),
	I5 is xor(I1, I3),
	I6 is xor(I2, I3),
	get_entry( Rows, [I1, I2], A ),
	get_entry( Rows, [I3, I4], B ),
	get_entry( Rows, [I1, I3], C ),
	get_entry( Rows, [I2, I5], D ),
	get_entry( Rows, [I2, I3], E ),
	get_entry( Rows, [I1, I6], F ),
	domain( [G,H,I], 0, 1),
	G #>= A + B - 1,
	G #=< A, 
	G #=< B,
	H #>= C + D - 1,
	H #=< C, 
	H #=< D,
	I #>= E + F - 1,
	I #=< E, 
	I #=< F,
	G + H + I #\= 1,
	G + H + I #\= 3.

% Simplicity conditions 
stop_certain_ideals(Rows,Indices) :-
	make_L1_inds( Indices, L1s),
	maplist( check_L1_makes_L0(Rows, Indices), L1s ).
make_L1_inds( Indices, L1s) :-
	maplist( root_to_L1(Indices), Indices, L1s ).
root_to_L1(Indices, Root, L1 ) :-
	findall(A, (member(A,Indices), B is (Root /\ A), dec2bin(B, Bin), sum(Bin,#=,S), (S mod 2) #= 1) , L1).
check_L1_makes_L0(Rows, Indices, L1) :-
	exclude( member_(L1), Indices, L0 ),
	maplist(is_made_check(Rows, L1), L0).
is_made_check( Rows, L1, A ) :-
	maplist( is_made_check_helper(Rows, A), L1, Entries ),
	sum( Entries, #>, 0 ).
is_made_check_helper( Rows, A, B, Entry ) :-
	C is xor(A, B),
	get_entry( Rows, [B,C], Entry ).
act_faithfully( Rows, Indices ) :-
	make_L1_inds( Indices, L1s ),
	maplist( check_L0_acts_faithfully( Rows, Indices ), L1s ).
check_L0_acts_faithfully( Rows, Indices, L1 ) :-
	exclude( member_(L1), Indices, L0 ),
	maplist(does_not_centralise(Rows, L1), L0).
does_not_centralise( Rows, L1, X ) :-
	maplist( does_not_centralise_helper( Rows, X ), L1, Entries ),
	sum( Entries, #>, 0 ).
does_not_centralise_helper( Rows, X, A, Entry ) :-
	get_entry( Rows, [A,X], Entry ).

%%% Symmetry breaking code %%%

% gl_2 is hardcoded due to its small size
get_gl2( GL2 ) :- 
	GL2 = [ [[1,0],[1,1]],[[1,1],[0,1]],[[1,1],[1,0]],[[0,1],[1,1]],[[0,1],[1,0]] ].

% For each pair of simple roots and each element of gl_2
% create a symmetry breaking constrint
break_gl2_symmetries( Vs, Rows, N ) :-
	get_gl2( GL2 ),
	numlist( N, SimpleInds ),
	findall( [A,B], ( member(A, SimpleInds), member(B, SimpleInds), A #< B ), Pairs ),
	maplist( break_gl2_symmetry( Vs, Rows, N, GL2), Pairs ).
break_gl2_symmetry( Vs, Rows, N, GL2, [J,K] ) :-
	maplist( add_to_gln_small( N, [J, K] ), GL2, SmolGLN ),
	get_roots(N,Roots), 
	make_powers(N,Powers),
	maplist( make_perm(Roots,Powers), SmolGLN, RowPerms ),
	maplist( break_symmetry( Vs, Rows), RowPerms ). 

% Create the subset of gl_n we create constraints for
add_to_gln_small( N, [J,K], Mat1, NewMat ) :- 
	M is N - 2,
	row_of_n_zeros( M, ZeroRow ),
	nth1( 1, Mat1, Row1 ),
	nth1( 2, Mat1, Row2 ),
	nth1( 1, Row1, A ),
	nth1( 2, Row1, B ),
	nth1( 1, Row2, C ),
	nth1( 2, Row2, D ),
	place_entry( J, A, ZeroRow, Row3 ),
	place_entry( K, B, Row3, Row4 ), % Place this at pos J
	place_entry( J, C, ZeroRow, Row5 ),
	place_entry( K, D, Row5, Row6 ), % Place this at pos K
	numlist( N, Indices ),
	findall( I, ( member( I, Indices ), I #\= J, I #\= K ), IdInds ),
	maplist( make_kth_row( Indices ), IdInds, Mat2 ),
	place_entry( J, Row4, Mat2, Mat3 ),
	place_entry( K, Row6, Mat3, NewMat ).

% Create the symmetry breaking constraint
% RowPerm is a permutation generated by an element of gl_n
% RowPerm is applied to the rows and columns of Row to obtain NewerRows
% Add the constraint that Rows is lexicographically lower than NewerRows
break_symmetry( Vs, Rows, RowPerm ) :-
	maplist(permute_rows(Rows), RowPerm, NewRows ),
	transpose(NewRows, TNewRows),
	maplist(permute_rows(TNewRows), RowPerm, NewerRows ),
	same_length( Vs, Ns ),
	append(NewerRows, Ns),
	lex_chain( [ Vs, Ns ] ).
%%%%%%

%%% 2. POST SEARCH SIMPLICTY CHECK %%%
% Search for certain non-trivial ideals for each table in Tables
% See paper for explanation
remove_non_simple_tables( Tables, SimpleTables ) :-
	include( simple_check, Tables, SimpleTables ).
	simple_check( Rows ) :-
	length( Rows, N ),
	numlist( N, Indices ),
	maplist( my_sum, Rows, RowSums ),
	sort( RowSums, UniqueRowSums ),
	findall(A, (member( B, UniqueRowSums ), A #= B + 1), IdealSizes ),
	maplist( check_for_ideals( Rows, Indices ), IdealSizes ).
sum_plus_one( Elems, Tot ) :-
	sum( Elems, #=, Tot0 ),
	Tot #= Tot0 + 1.
check_for_ideals(Rows, Indices, IdealSize ) :-
	findall( A, ( member(A, Indices), nth1(A, Rows, Row), sum(Row, #=, S), S #= IdealSize - 1), CorrectRankRoots),
	findall( A, ( length( A, IdealSize), subset_set( A, CorrectRankRoots ) ), PossibleIdeals ),
	maplist( check_ideal( Rows ), PossibleIdeals ).
check_ideal( Rows, PossibleIdeal ) :-
	maplist( check_ideal_helper(Rows, PossibleIdeal), PossibleIdeal, Mat ),
	append( Mat, Ms ),
	sum( Ms, #=, Tot ),
	length( PossibleIdeal, N ),
	M is N^2 - N,
	Tot #\= M.
check_ideal_helper( Rows, PossibleIdeal, X, Row ) :-
	maplist( check_ideal_helper_2( Rows, X), PossibleIdeal, Row ).
check_ideal_helper_2( _, X, X, Entry ) :-
	Entry #= 0.
check_ideal_helper_2( Rows, X, Y, Entry ) :-
	X #\= Y,
	I is xor(X,Y),
	get_entry( Rows, [I,X], Entry).
%%%%%%

%%% 3. LEX REDUCE TABLES %%%
% Remove all tables from SimpleTables that are not in their lexicographically min form
lex_reduce_tables( N, SimpleTables, ReducedTables ) :-
get_roots( N, Roots ),
make_powers( N, Powers ),
% For each table calculate its row sums and partition SimpleTables based on this
populate_row_sum_tables_list( SimpleTables, TablesList ),
% For each list of tables remove tables not in lexicographically min form
maplist( lex_reduce(N, Roots, Powers), TablesList, ReducedTablesList ),
% Combine ReducedTablesList in to ReducedTables
append( ReducedTablesList, ReducedTables ).

% Recursively generate ReducedList from List
% 1. Add the element of List with min lexicographical order to ReducedList, call this MinTable
% 2. Remove all elements equivalement to MinTable via row permutations from List
% 3. Repeat 1. and 2. until List is empty 
lex_reduce( N, Roots, Powers, List, ReducedList ) :-
	lex_reduce( N, Roots, Powers, List, [], ReducedList).
lex_reduce( _, _, _, [], Mins, ReducedList) :- 
	ReducedList = Mins.
lex_reduce( _, _, _, List, Mins, ReducedList) :-
	length( List, 1 ),
	append( Mins, List, ReducedList ).
lex_reduce( N, Roots, Powers, List, Mins, ReducedList) :-
	length( List, K ),
	K #> 1,
	min_member( my_lex, MinTable, List ),
	sorted_row_sums( MinTable, SortedRowSums ),
	exclude( can_permute_dispatcher( N, Roots, Powers, SortedRowSums, MinTable), List,  NewList ),
	append( Mins, [MinTable], NewMins ),
	lex_reduce( N, Roots, Powers, NewList, NewMins, ReducedList).

% Calculate row sums and partition with respect to row sums
populate_row_sum_tables_list( SimpleTables, TablesList ) :-
	maplist( full_sorted_row_sums, SimpleTables, TooManyRowSums ),
	list_to_ord_set( TooManyRowSums, RowSumsList ),
	maplist( filter_by_row_sum(SimpleTables), RowSumsList, TablesList ).
filter_by_row_sum( SimpleTables, RowSums, Tables ) :-
	findall( Table, ( member( Table, SimpleTables ), full_sorted_row_sums( Table, RowSums ) ), Tables ).

% lexicographical order for arrays
my_lex( A, B ) :-
	append( A, As ),
	append( B, Bs ),
	lex_chain( [As, Bs] ).
%%%%%%

%%% 4. TORAL SWITCHING %%%
perform_toral_switchings( N, Tables, SwitchedTables ) :-
	M is 2^N-1,
	K is 2^(N-1),
	numlist( M, Indices ),
	get_roots( N,Roots ),
	make_powers( N, Powers ), 
	% For each table in Tables find all the basis elements which are nilpotent
	maplist(locate_basis_nilps( Indices, K ), Tables, NilpIndicesList ),
	% Make a torus which all switches will be made with respect to
	make_base_torus( BaseTorus, N, Indices ),
	findall( [A,B], ( member(A, Indices), member(B, Indices), A #< B ), Pairs ), 
	% Make a torus switch for each nilpotent basis element
	maplist( make_torus_switches( Indices, Pairs, Roots, BaseTorus ), Tables, NilpIndicesList, SwitchedTablesList ),
	% Lexicographically reduce each switched table
	once(canonical_order_tables_list( N, Roots, Powers, Tables, SwitchedTablesList, OrderedTablesList )), 
	% Make the graph based on toral switchings. See paper for details
	once(make_adjacency_mat( Tables, OrderedTablesList, AdjacencyMat )),
	adjacency_mat_to_ugraph( AdjacencyMat, SwitchingGraph ),
	conn_comps( SwitchingGraph, Comps ),
	% Return one table per connected componenet of the graph
	maplist( get_min_from_comp( Tables ), Comps, SwitchedTables ).

% Finds nilpotent basis elements
locate_basis_nilps( Indices, K, Rows, NilpIndices ) :-
	include( locate_basis_nilps_helper( Rows, K ), Indices, NilpIndices ).
locate_basis_nilps_helper( Rows, K, I ) :-
	nth1( I, Rows, Row ), 
	sum( Row, #<, K ).

% Makes standard torus to be used in all switchings
make_base_torus( BaseTorus, N, Indices ) :-
	length( BaseTorus, N ),
	numlist( N, SmolIndices ),
	maplist( make_base_torus_row( Indices ), SmolIndices, BaseTorus ).
make_base_torus_row( Indices, RowIndex, BaseTorusRow ) :-
	RootIndex is 2^(RowIndex - 1),
	maplist( make_base_torus_entry( RootIndex ), Indices, BaseTorusRow ).
make_base_torus_entry( RootIndex, ColIndex, Entry ) :-
	X is ( RootIndex /\ ColIndex ),
	make_base_torus_entry_helper( X, Entry ).
	make_base_torus_entry_helper( 0, 0 ).
make_base_torus_entry_helper( X, 1 ) :-
	X #\= 0.

% Main toral switching code
make_torus_switches( _, _, _, _, _, [], SwitchedTables ) :-
	SwitchedTables = []. % no nilpotent basis elements to switch with respect to
make_torus_switches( Indices, Pairs, Roots, BaseTorus, Table, NilpIndices, SwitchedTables ) :-
	length( NilpIndices, K ),
	K #> 0,
	maplist( switch_torus( Indices, Pairs, Table, BaseTorus, Roots), NilpIndices, SwitchedTables ).
switch_torus( Indices, Pairs, Table, BaseTorus, Roots, NilpIndex, SwitchedTable ) :-
	maplist( modify_toral_element_checker( NilpIndex ), BaseTorus, ModifyToralElementChecker ),
	maplist( make_new_basis( Indices, Table, BaseTorus, NilpIndex, ModifyToralElementChecker), Roots, NewBasis ),
	make_new_table( Indices, Pairs, Table, NewBasis, SwitchedTable ).

% Checks wether the toral element in the switched torus has been modified
modify_toral_element_checker( NilpIndex, ToralRow, ModifyToralElement ) :-
	nth1( NilpIndex, ToralRow, ModifyToralElement ).

% NewBasisElement is the basis element in the rootspace with respect to Root after the toral switch
% It is expressed as a list in terms of the original thin basis
make_new_basis( Indices, Table, BaseTorus, NilpIndex, ModifyToralElementChecker, Root, NewBasisElement ) :-
	same_length( Table, NewBasisElement ),
	domain( NewBasisElement, 0, 1),
	once( make_new_basis_helper( Indices, Table, BaseTorus, NilpIndex, ModifyToralElementChecker, Root, NewBasisElement ) ),
	labeling( [], NewBasisElement ) .
make_new_basis_helper( Indices, Table, BaseTorus, NilpIndex, ModifyToralElementChecker, Root, NewBasisElement ) :-
	sum( NewBasisElement, #>, 0 ),
	maplist( apply_root_value( Indices, Table, NilpIndex, NewBasisElement), BaseTorus, ModifyToralElementChecker, Root ).

% Apply a new toral element on NewBasisElement and kill or stabalise it depending on RootVal
% NewBasisElement is determined by the constraints created here
apply_root_value( Indices, Table, NilpIndex, NewBasisElement, ToralRow, ModifyToralElement, RootVal ) :-
	maplist( times, ToralRow, NewBasisElement, TorusOnElement),
	nth1( NilpIndex, Table, NilpRow ),
	maplist( nilp_on_element(NilpRow, NewBasisElement, NilpIndex), Indices, NilpOnElement ),
	maplist( times(ModifyToralElement), NilpOnElement, ScaledNilpOnElement ),
	maplist( sum_mod_2, TorusOnElement, ScaledNilpOnElement, ActedOnElement ),
	maplist( times(RootVal), NewBasisElement, ScaledNewBasisElement ),
	maplist( eq, ScaledNewBasisElement, ActedOnElement).
nilp_on_element(_, _, NilpIndex, NilpIndex, NilpOnElementEntry ) :- 
	NilpOnElementEntry #= 0.
nilp_on_element(NilpRow, NewBasisElement, NilpIndex, Index, NilpOnElementEntry ) :-
	NilpIndex #\= Index,
	NewIndex is xor(NilpIndex, Index),
	nth1( NewIndex, NewBasisElement, V),
	nth1( NewIndex, NilpRow, NilpOnElementEntry0 ),
	NilpOnElementEntry #= NilpOnElementEntry0 * V.

% Given the new thin basis NewBasis calculate its thin table 
make_new_table( Indices, Pairs, Table, NewBasis, NewTable ) :-
	maplist( set_diagonal( NewTable ), Indices ),
	maplist( intify_basis(Indices), NewBasis, IntBasis ),
	same_length( Table, NewTable ),
	maplist( same_length(NewTable), NewTable ),
	maplist( make_new_table_entry( Indices, Table, IntBasis, NewTable ), Pairs ).
make_new_table_entry( Indices, Table, IntBasis, NewTable, [I1, I2] ) :-
	I3 is xor( I1, I2 ),
	nth1( I1, IntBasis, Ints1 ),
	nth1( I2, IntBasis, Ints2 ),
	nth1( I3, IntBasis, Ints3 ),
	nth1( 1, Ints3, MainInt ),
	maplist( make_xor_pair(MainInt), Indices, XORPairs ),
	include( filter_xor_pairs( Ints1, Ints2 ), XORPairs, FilteredPairs ), 
	get_entries( Table, FilteredPairs, Entries ),
	sum(Entries, #=, Tot),
	Val #= ( Tot mod 2 ),
	get_entry( NewTable, [I1, I2], Val ),
	get_entry( NewTable, [I2, I1], Val ).
% The below predicates appear as they are due to Sicstus not having a version
% of xor compatible with  the # operation in CLPFD. 
filter_xor_pairs( Ints1, Ints2, [I, J ] ) :-
	member(I, Ints1),
	member(J, Ints2).
make_xor_pair( A, B, [B, C] ) :-
	C is xor( A, B).
intify_basis( Indices, BasisRow, Ints ) :-
	include( intify_basis_helper(BasisRow), Indices, Ints ).
intify_basis_helper( BasisRow, I ) :-
	nth1( I, BasisRow, 1 ).
set_diagonal( NewTable, I ) :-
	get_entry( NewTable, [I,I], 0 ).

% Lexicographically reduce switched tables. Similar methods to 3. LEX REDUCE TABLES 
canonical_order_tables_list( N, Roots, Powers, Tables, SwitchedTablesList, OrderedTablesList) :-
	maplist( make_table_row_sum_pair, Tables, TablesWithSums ),
	maplist( canonical_order_rows0( N, Roots, Powers, TablesWithSums ), SwitchedTablesList, OrderedTablesList ).
make_table_row_sum_pair( Table, [X, Y] ) :-
	X = Table,
	full_sorted_row_sums( Table, Y ).
canonical_order_rows0( _, _, _, _, [], OrderedTables) :-
	OrderedTables = [].
canonical_order_rows0( N, Roots, Powers, SimpleTablesWithSums, Tables, OrderedTables) :-
	length(Tables, K),
	K #> 0,
	maplist(canonical_order_rows( N, Roots, Powers, SimpleTablesWithSums ), Tables, OrderedTables).
canonical_order_rows( N, Roots, Powers, SimpleTablesWithSums, Rows, OrderedRows) :-
	full_sorted_row_sums( Rows, RowSums),
	sort( RowSums, SortedRowSums),
	findall( Table, ( member(TableWithSums, SimpleTablesWithSums), TableWithSums = [ Table, RowSums ] ), SimpleTables ),
	first_sol( lex_reduce_to_simple( N, Roots, Powers, Rows, SortedRowSums ), SimpleTables, OrderedRows ).
lex_reduce_to_simple( N, Roots, Powers, Rows, SortedRowSums, Table ) :-
	can_permute_dispatcher( N, Roots, Powers, SortedRowSums, Rows, Table).

% Make the toral switching graph
% If a table T1 switches to T2 then T2 might not have switched to T1
% Thus we calculate a directed graph first
% We then remove the directions on edges 
make_adjacency_mat( Tables, OrderedTablesList, AdjacencyMat ) :-
	length( Tables, M ),
	numlist( M, Indices ),
	maplist( make_adjacency_row( Tables, OrderedTablesList, Indices ), Indices, Mat ),
	transpose( Mat, TMat),
	maplist(maplist( my_max ), Mat, TMat, AdjacencyMat).
make_adjacency_row( Tables, OrderedTablesList, Indices, RowIndex, AdjacencyRow ) :-
	maplist( make_adjacency_entry( Tables, OrderedTablesList, RowIndex ), Indices, AdjacencyRow ).
make_adjacency_entry( _, _, RowIndex , RowIndex, AdjacencyEntry ) :-
	AdjacencyEntry #= 0.
make_adjacency_entry( Tables, OrderedTablesList, RowIndex , ColIndex, AdjacencyEntry ) :-
	RowIndex #\= ColIndex,
	nth1( RowIndex, Tables, Table ),
	nth1( ColIndex, OrderedTablesList, TableList ),
	set_adjacency_entry( Table, TableList, AdjacencyEntry ).
set_adjacency_entry( Table, TableList, AdjacencyEntry ) :-
	member( Table, TableList ),
	AdjacencyEntry #= 1.
set_adjacency_entry( Table, TableList, AdjacencyEntry ) :-
	(\+ member( Table, TableList )),
	AdjacencyEntry #= 0.
	% Convert an adjacency matrix to a ugraph
adjacency_mat_to_ugraph( A, G ) :-
	length( A, N ),
	numlist(N, Indices),
	maplist( collect_nbrs(A,Indices), Indices, NbrsList ), 
	maplist( make_vertex, Indices, NbrsList, G ).
make_vertex( Index, Nbrs, V ) :-
	V = Index-Nbrs. % This is the ugraph notation for defining a vertex and its neighbours
% Find the neighbours (Nbrs) of a vertex (Index)
collect_nbrs(A, Indices, Index, Nbrs ) :-
	nth1( Index, A, Row ),
	findall( X, ( member( X, Indices ), nth1( X, Row, 1 ) ), Nbrs).

% Calculte the connected componenets of a graph.
% We make use of transitive_closure from ugraphs
conn_comps( G, Comps ) :-
	length( G, N), 
	numlist(N, Indices),
	transitive_closure(G, C),
	maplist( get_closure(C), Indices, RepeatComps  ),
	setof( Comp, member( Comp, RepeatComps ), Comps ).
get_closure(C, Index, Closure ) :-
	neighbours( Index, C, PreClosure),
	add_self( Index, PreClosure, NotSortedClosure),
	sort( NotSortedClosure, Closure).
add_self( Index, A, B) :-
	member(Index, A),
	B = A.
add_self( Index, A, B) :-
	(\+ member( Index, A )),
	append( [A, [Index]], B ).
get_min_from_comp( Tables, Comp, MinTable ) :-
	get_values( Tables, Comp, CompTables ),
	min_member( my_lex, MinTable, CompTables ).
%%%%%%

%%% 5. SYMMETRY PREDICATES %%%
% Code used in symmetry breaking and lexicographically reducing tables

% Make a permutation of rows based on an element of gl_n acting on the simple roots
make_perm(Roots,Powers, Mat, RowPerm ) :-
	maplist( perm_it(Powers, Mat), Roots, RowPerm ).
perm_it(Powers, Mat, Root, Entry ) :-
	act_mat(Mat, Root, RootOut),
	bin_2_dec(Powers, RootOut, Entry).

% Apply a permutation
permute_rows(Rows, PermIndex, NewRow ) :-
	nth1( PermIndex, Rows, NewRow).

% The predicate can_permute_dispatcher checks if Table1 can be permuted in to Table2
% This is used in both lex reducing and toral switching
can_permute_dispatcher( N, Roots, Powers, SortedRowSums, Table1, Table2 ) :-
	length( Mat, N ),
	maplist(same_length(Mat), Mat),
	once(can_permute( N, Roots, Powers, Table1, Table2, Mat, SortedRowSums ) ).
can_permute( N, Roots, Powers, T1, T2, Mat, SortedRowSums ) :-
	length( Mat, N ),
	append( Mat, Ms), 
	domain( Ms, 0, 1),
	% Use the row sums of T1 and T2 to crete constraints on which rows of T1 are mapped to which rows of T2
	partition_by_row_sums( T1, SortedRowSums, Partition1 ),
	partition_by_row_sums( T2, SortedRowSums, Partition2 ),
	maplist( map_partitions(N, Mat, Powers), Partition1, Partition2 ),
	make_perm( Roots, Powers, Mat, RowPerm ),
	maplist(permute_rows(T2), RowPerm, NewRows ),
	transpose(NewRows, TNewRows),
	maplist(permute_rows(TNewRows), RowPerm, T1 ).
map_partitions( N, Mat, _, L1, L2 ) :-
	length( L1, K ),
	K #= 1,
	Rank is N,
	nth1( 1, L1, N1 ),
	nth1( 1, L2, N2 ),
	dec_2_bin( Rank, N1, B1 ),
	dec_2_bin( Rank, N2, B2 ),
	act_mat( Mat, B1, B2 ).
map_partitions( N, Mat, Powers, L1, L2 ) :-
	length( L1, K ),
	K #> 1,
	Rank is N,
	maplist( dec_2_bin( Rank ), L1, Domain ),
	maplist( set_range( Mat, Powers, L2 ), Domain ).
set_range( Mat, Powers, Range, B1 ) :-
	act_mat( Mat, B1, B2 ),
	bin_2_dec( Powers, B2, D ),
	element( _, Range, D ).
partition_by_row_sums( Mat, SortedRowSums, Partition ) :-
	maplist( has_row_sum(Mat), SortedRowSums, Partition ).
has_row_sum( Mat, S, L ) :-
	findall( I, ( nth1( I, Mat, R ), sum( R, #=, S ) ), L ).

% row sums with no repeats
sorted_row_sums( M, Sorted ) :-
	maplist( my_sum, M, Sums ),
	sort( Sums, Sorted ).
% row sums including repeats. Sicstus needs samsort importing
full_sorted_row_sums( M, Sorted ) :-
	maplist( my_sum, M, Sums ),
	samsort( Sums, Sorted ).
%%%%%%

%%%%%%

%%% 6. UTILITY PREDICATES %%%
% generic predicates used throughout

% Basic computations
plus(X,Y,Z) :- X+Y #= Z.
eq(A,B) :- A #= B.
equal_to_zero( X ) :- X #= 0.
sum_mod_2( A, B, C ) :- ((A + B) mod 2 ) #= C.
times(A,B,C) :- A*B  #= C.

% equivalent to sum( Elems, #=, Tot )
my_sum( Elems, Tot ) :-
	sum( Elems, #=, Tot ).

% C is the max of A and B
my_max( A, B, C ) :-
	A #< B,
	C #= B.
my_max( A, B, C ) :-
	A #>= B,
	C #= A.

% Get entries from lists and arrays
get_value( Row, Index, Value ) :-
	nth1( Index, Row, Value).
get_values( Row, Indices, Values ) :-
	maplist( get_value(Row), Indices, Values ).
get_entry( Rows, [RowIndex, ColIndex], Entry ) :-
	nth1( RowIndex, Rows, Row),
	nth1( ColIndex, Row, Entry).
get_entries( Rows, Indices, Entries ) :-
	maplist( get_entry(Rows), Indices, Entries ).

% set an element of a list to be 0
set_value_to_zero( Index, Row ) :-
	nth1( Index, Row, 0).

% Powers is a list of powers of 2
make_powers(N, Powers) :-
	make_powers0(N, Roots),
	reverse(Roots, Powers).
make_powers0(N, Roots) :-
	numlist(N, Indices),
	maplist(make_powers_helper, Indices, Roots).  
make_powers_helper(Ind, Root) :-
	Val is 2^(Ind - 1),
	Root #= Val.

% binary and decimal conversion
dec_2_bin(Rank,N,L) :-
	dec2bin(N, L0),
	length(L0,K),
	length(L,Rank),
	M #= Rank - K,
	row_of_n_zeros( M, Zs ),
	append(Zs, L0, L).
	dec2bin(0,[0]).
	dec2bin(1,[1]).
dec2bin(N,L):- 
	N > 1,
	X #= ( N mod 2 ),
	Y #= (N // 2),  
	dec2bin(Y,L1), 
	append(L1, [X], L).
bin_2_dec(Powers, L, N) :-
	maplist(times, Powers, L, P),
	sum(P,#=,N).

member_( L, E ) :- 
	member( E, L ).

% generate roots as elements of GF(2)^N
get_roots(N, Roots) :-
	findall(L,( length(L,N), domain( L, 0, 1 ), labeling( [], L)  ),RootsZ),
	RootsZ=[_|Roots].

% Mat acting on VecIn gives VecOut
act_mat(Mat, VecIn, VecOut) :- % NB VecIn cannot have variables
	maplist( my_scalar_prod( VecIn ), Mat, VecOut).
my_scalar_prod( V1, M1, X) :- 
	same_length( V1, Prods ),
	maplist( times, V1, M1, Prods ),
	sum( Prods, #=, Y ), 
	(Y mod 2) #= X.

% make a list of all zeroes
row_of_n_zeros( N, Row ) :-
	length( Row, N ),
	maplist( equal_to_zero, Row ).

% KthRow has the same length as Indices, a 1 in position K, and zeroes elsewhere
make_kth_row( Indices, K, KthRow ) :-
	maplist( set_kth_row( K ), Indices, KthRow ).
	set_kth_row(  K, K, Entry ) :- Entry #= 1.
	set_kth_row(  K, Index, Entry ) :- Index #\= K, Entry #= 0. 

% adding Entry to X at position I obtains Y
place_entry( Ind, Entry, X, Y ) :- 
	L is Ind - 1,
	length( A, L ),
	append( A, B, X ),
	append( [ A, [Entry], B ], Y ).

% recursively check if a set is a subset of another set
subset_set([], _).
subset_set([X|Xs], S) :-
	append(_, [X|S1], S),
	subset_set(Xs, S1).

% Sol is the first element of List to satisfy Goal
first_sol(Goal, List, Sol) :-
	first_sol_(List, Goal, Sol).
	first_sol_([], _, []).
first_sol_([X1|Xs1], P, Sol) :-
	(   call(P, X1) ->  Sol = X1
	;   first_sol_(Xs1, P, Sol )
	).

% Sictus does not have writeln and one of the authors is fond of it
writeln( Stream ) :-
	write( Stream ),
	write('\n').

% Sicstus needs maplist/4 defining manually
maplist(Pred, Ws, Xs, Ys, Zs) :-
	( foreach(W,Ws),
	foreach(X,Xs),
	foreach(Y,Ys),
	foreach(Z,Zs),
	param(Pred)
	do call(Pred, W, X, Y, Z)
	).
%%%%%%

\end{lstlisting}

\section{SWI monotonic prolog code}
\begin{lstlisting}
:- use_module(library(clpfd)), set_prolog_flag(clpfd_monotonic, true).

% The following code runs in SWI prolog with monotonic enabled. 
% See the above Sicstus code for necessary comments 

%%% 0. THE MAIN PREDICATE %%%
reduced_thin_search( N, ThinTables ) :-
	M is 2**N-1,
	findall( Rows, (thin_search( Vs, N, M, Rows), label(Vs) ), FirstTables ),
	remove_non_simple_tables( FirstTables, SimpleTables ),
	lex_reduce_tables( N, SimpleTables, ReducedTables ),
	once(perform_toral_switchings( N, ReducedTables, ThinTables )).
%%%%%%

%%% 1. INITIAL THIN SEARCH %%%
thin_search( Vs, N, M, Rows) :-
	length(Rows, M),
	maplist(same_length(Rows),Rows),
	append(Rows, Vs), Vs ins 0..1,
	numlist(M, Indices),
	transpose(Rows, Rows),
	maplist( set_value_to_zero, Indices, Rows),
	stop_certain_ideals(Rows,Indices),
	act_faithfully(Rows,Indices),
	jacobi_identity_full( Indices, Rows),
	break_gl2_symmetries( Vs, Rows, N ).

% The Jacobi Identity 
jacobi_identity_full( Indices, Rows) :-
	findall( [A,B,C], ( append( [_,[A],_,[B],_,[C],_], Indices ),  A #\= (B xor C) ), Triples ),
	maplist( jacobi_identity( Rows ), Triples ).
jacobi_identity( Rows, [I1,I2,I3] ) :-
	#(I4) #= (#(I1) xor #(I2)),
	#(I5) #= (#(I1) xor #(I3)),
	#(I6) #= (#(I2) xor #(I3)),
	get_entry( Rows, [I1, I2], A ),
	get_entry( Rows, [I3, I4], B ),
	get_entry( Rows, [I1, I3], C ),
	get_entry( Rows, [I2, I5], D ),
	get_entry( Rows, [I2, I3], E ),
	get_entry( Rows, [I1, I6], F ),
	[G,H,I] ins 0..1,
	#(G) #>= #(A) + #(B) - 1,
	#(G) #=< #(A), 
	#(G) #=< #(B),
	#(H) #>= #(C) + #(D) - 1,
	#(H) #=< #(C), 
	#(H) #=< #(D),
	#(I) #>= #(E) + #(F) - 1,
	#(I) #=< #(E), 
	#(I) #=< #(F),
	#(G) + #(H) + #(I) #\= 1,
	#(G) + #(H) + #(I) #\= 3.

% Simplicity conditions 
stop_certain_ideals(Rows,Indices) :-
	make_L1_inds( Indices, L1s),
	maplist( check_L1_makes_L0(Rows, Indices), L1s ).
make_L1_inds( Indices, L1s) :-
	maplist( root_to_L1(Indices), Indices, L1s ).
root_to_L1(Indices, Root, L1 ) :-
	findall( A, ( member(A,Indices), #(B) #= (#(Root) /\ #(A)), dec2bin(B,Bin ), sum(Bin,#=,#(S)), (#(S) mod 2) #= 1) L1).
check_L1_makes_L0(Rows, Indices, L1) :-
	exclude( member_(L1), Indices, L0 ),
	maplist(is_made_check(Rows, L1), L0)
is_made_check( Rows, L1, A ) :-
	maplist( is_made_check_helper(Rows, A), L1, Entries ),
	sum( Entries, #>, 0 ).
is_made_check_helper( Rows, A, B, Entry ) :-
	#(C) #= (#(A) xor #(B)),
	get_entry( Rows, [B,C], Entry ).
act_faithfully( Rows, Indices ) :-
	make_L1_inds( Indices, L1s ),
	maplist( check_L0_acts_faithfully( Rows, Indices ), L1s ).
check_L0_acts_faithfully( Rows, Indices, L1 ) :-
	exclude( member_(L1), Indices, L0 ),
	maplist(does_not_centralise(Rows, L1), L0).
does_not_centralise( Rows, L1, X ) :-
	maplist( does_not_centralise_helper( Rows, X ), L1, Entries ),
	sum( Entries, #>, 0 ).
does_not_centralise_helper( Rows, X, A, Entry ) :-
	get_entry( Rows, [A,X], Entry ).

%%% Symmetry breaking code %%%
get_gl2( GL2 ) :- GL2 = [ [[1,0],[1,1]],[[1,1],[0,1]],[[1,1],[1,0]],[[0,1],[1,1]],[[0,1],[1,0]] ].

break_gl2_symmetries( Vs, Rows, N ) :-
	get_gl2( GL2 ),
	numlist( N, SimpleInds ),
	findall( [A,B], ( member(A, SimpleInds), member(B, SimpleInds), #(A) #< #(B) ), Pairs ), 
	get_roots( N, Roots ), 
	make_powers( N, Powers ),
	numlist( N, Indices ),
	maplist( break_gl2_symmetry( Vs, Rows, N, Indices, Roots, Powers, GL2), Pairs ).
break_gl2_symmetry( Vs, Rows, N, Indices, Roots, Powers, GL2, [J,K] ) :-
	maplist( add_to_gln_small( N, Indices, [J, K] ), GL2, SmallGLN ),
	maplist( make_perm(Roots,Powers), SmallGLN, RowPerms ),
	maplist( break_symmetry( Vs, Rows), RowPerms ). 

add_to_gln_small( N, Indices, [J,K], Mat1, NewMat ) :- 
	M is N - 2,
	row_of_n_zeros( M, ZeroRow ),
	nth1( 1, Mat1, Row1 ),
	nth1( 2, Mat1, Row2 ),
	nth1( 1, Row1, A ),
	nth1( 2, Row1, B ),
	nth1( 1, Row2, C ),
	nth1( 2, Row2, D ),
	place_entry( J, A, ZeroRow, Row3 ),
	place_entry( K, B, Row3, Row4 ), % Place this at pos J
	place_entry( J, C, ZeroRow, Row5 ),
	place_entry( K, D, Row5, Row6 ), % Place this at pos K
	exclude( member_([J,K]), Indices, IdInds ),
	maplist( make_kth_row( Indices ), IdInds, Mat2 ),
	place_entry( J, Row4, Mat2, Mat3 ),
	place_entry( K, Row6, Mat3, NewMat ).

break_symmetry( Vs, Rows, RowPerm ) :-
	maplist(permute_rows(Rows), RowPerm, NewRows ),
	transpose(NewRows, TNewRows),
	maplist(permute_rows(TNewRows), RowPerm, NewerRows ),
	same_length( Vs, Ns ),
	append(NewerRows, Ns),
	lex_chain( [ Vs, Ns ] ).
%%%%%%

%%% 2. POST SEARCH SIMPLICTY CHECK %%%
remove_non_simple_tables( Tables, SimpleTables ) :-
	include( simple_check, Tables, SimpleTables ).
simple_check( Rows ) :-
	length( Rows, N ),
	numlist( N, Indices ),
	maplist( my_sum, Rows, RowSums ),
	sort( RowSums, UniqueRowSums ),
	maplist( plus(1), UniqueRowSums, IdealSizes ),
	maplist( check_for_ideals( Rows, Indices ), IdealSizes ).
check_for_ideals(Rows, Indices, IdealSize ) :-
	S is IdealSize - 1,
	include( check_for_ideal_helper( Rows, S ), Indices, CorrectRankRoots ),
	findall( I, ( length( I, IdealSize), subset_set( I, CorrectRankRoots ) ), PossibleIdeals ),
	maplist( check_ideal( Rows ), PossibleIdeals ).
check_for_ideal_helper( Rows, S, I ) :-
	nth1( I, Rows, Row ), 
	sum( Row, #=, #(S) ).
check_ideal( Rows, PossibleIdeal ) :-
	maplist( check_ideal_helper( Rows, PossibleIdeal ), PossibleIdeal, Mat ),
	append( Mat, Ms ),
	sum( Ms, #=, #(Tot) ),
	length( PossibleIdeal, N ),
	M is N^2 - N,
	#(Tot) #\= #(M).
check_ideal_helper( Rows, PossibleIdeal, X, Row ) :-
	maplist( check_ideal_helperer( Rows, X ), PossibleIdeal, Row ).
check_ideal_helperer( _, X, X, Entry ) :-
	#(Entry) #= 0.
check_ideal_helperer( Rows, X, Y, Entry ) :-
	#(X) #\= #(Y),
	#(I) #= ( #(X) xor #(Y) ),
	get_entry( Rows, [I,X], Entry).
	%%%%%%

%%% 3. LEX REDUCE TABLES %%%
lex_reduce_tables( N, SimpleTables, ReducedTables ) :-
	populate_row_sum_tables_list( SimpleTables, TablesList ),
	maplist( lex_reduce(N), TablesList, ReducedTablesList ),
	append( ReducedTablesList, ReducedTables ).

lex_reduce( N, List, ReducedList ) :-
	lex_reduce( N, List, [], ReducedList).
lex_reduce( _, [], Mins, ReducedList) :- 
	ReducedList = Mins.
lex_reduce( _, List, Mins, ReducedList) :-
	length( List, 1 ),
	append( Mins, List, ReducedList ).
lex_reduce( N, List, Mins, ReducedList) :-
	length( List, K ),
	#(K) #> 1,
	min_member( my_lex, MinTable, List ),
	sorted_row_sums( MinTable, SortedRowSums ),
	get_roots( N, Roots ),
	make_powers( N, Powers ),
	findall( Table, (member(Table, List), \+ can_permute_dispatcher( N, Roots, Powers, MinTable, Table, SortedRowSums) ), NewList ),
	append( Mins, [MinTable], NewMins ),
	lex_reduce( N, NewList, NewMins, ReducedList).

populate_row_sum_tables_list( SimpleTables, TablesList ) :-
	maplist( full_sorted_row_sums, SimpleTables, TooManyRowSums ),
	list_to_set( TooManyRowSums, RowSumsList ),
	maplist( filter_by_row_sum(SimpleTables), RowSumsList, TablesList ).
filter_by_row_sum( SimpleTables, RowSums, Tables ) :-
	include( filter_by_row_sum_helper(RowSums), SimpleTables, Tables).

my_lex( A, B ) :-
	append( A, As ),
	append( B, Bs ),
	lex_chain( [As, Bs] ).
%%%%%%

%%% 4. TORAL SWITCHING %%%
perform_toral_switchings( N, Tables, SwitchedTables ) :-
	M is 2**N-1,
	K is 2**(N-1),
	numlist( M, Indices ),
	get_roots( N, Roots ),
	make_powers( N, Powers ),    
	maplist(locate_basis_nilps( Indices, K ), Tables, NilpIndicesList ),
	make_base_torus( BaseTorus, N, Indices ),
	findall( [A,B], ( member(A, Indices), member(B, Indices), #(A) #< #(B) ), Pairs ), 
	maplist( make_torus_switches( Indices, Pairs, Roots, BaseTorus ), Tables, NilpIndicesList, SwitchedTablesList ),
	once(canonical_order_tables_list( N, Roots, Powers, Tables, SwitchedTablesList, OrderedTablesList )), 
	once(make_adjacency_mat( Tables, OrderedTablesList, AdjacencyMat )),
	adjacency_mat_to_ugraph( AdjacencyMat, SwitchingGraph ),
	conn_comps( SwitchingGraph, Comps ),
	maplist( get_min_from_comp( Tables ), Comps, SwitchedTables ).

locate_basis_nilps( Indices, K, Rows, NilpIndices ) :-
	include( locate_basis_nilps_helper( Rows, K ), Indices, NilpIndices ).
locate_basis_nilps_helper( Rows, K, I ) :-
	nth1( I, Rows, Row ), 
	sum( Row, #<, #(K) ).

make_base_torus( BaseTorus, N, Indices ) :-
	length( BaseTorus, N ),
	numlist( N, SmallIndices ),
	maplist( make_base_torus_row( Indices ), SmallIndices, BaseTorus ).
make_base_torus_row( Indices, RowIndex, BaseTorusRow ) :-
	RootIndex is 2**(RowIndex - 1),
	maplist( make_base_torus_entry( RootIndex ), Indices, BaseTorusRow ).
make_base_torus_entry( RootIndex, ColIndex, Entry ) :-
	( #(RootIndex) /\ #(ColIndex) ) #= 0,
	#(Entry) #= 0.
make_base_torus_entry( RootIndex, ColIndex, Entry ) :-
	( #(RootIndex) /\ #(ColIndex) ) #= #(RootIndex),
	#(Entry) #= 1.

make_torus_switches( _, _, _, _, _, [], SwitchedTables ) :-
	SwitchedTables = [].
make_torus_switches( Indices, Pairs, Roots, BaseTorus, Table, NilpIndices, SwitchedTables ) :-
	length( NilpIndices, K ),
	#(K) #> 0,
	maplist( switch_torus( Indices, Pairs, Table, BaseTorus, Roots), NilpIndices, SwitchedTables ).
switch_torus( Indices, Pairs, Table, BaseTorus, Roots, NilpIndex, SwitchedTable ) :-
	maplist( modify_toral_element_checker( NilpIndex ), BaseTorus, ModifyToralElementChecker ),
	maplist( make_new_basis( Indices, Table, BaseTorus, NilpIndex, ModifyToralElementChecker), Roots, NewBasis ),
	make_new_table( Indices, Pairs, Table, NewBasis, SwitchedTable ).

modify_toral_element_checker( NilpIndex, ToralRow, ModifyToralElement ) :-
	nth1( NilpIndex, ToralRow, ModifyToralElement ).

make_new_basis( Indices, Table, BaseTorus, NilpIndex, ModifyToralElementChecker, Root, NewBasisElement ) :-
	same_length( Table, NewBasisElement ),
	NewBasisElement ins 0..1,
	once( make_new_basis_helper( Indices, Table, BaseTorus, NilpIndex, ModifyToralElementChecker, Root, NewBasisElement ) ),
	label(NewBasisElement).
make_new_basis_helper( Indices, Table, BaseTorus, NilpIndex, ModifyToralElementChecker, Root, NewBasisElement ) :-
	sum( NewBasisElement, #>, 0 ),
	maplist( apply_root_value( Indices, Table, NilpIndex, NewBasisElement), BaseTorus, ModifyToralElementChecker, Root ).

apply_root_value( Indices, Table, NilpIndex, NewBasisElement, ToralRow, ModifyToralElement, RootVal ) :-
	maplist( times, ToralRow, NewBasisElement, TorusOnElement),
	nth1( NilpIndex, Table, NilpRow ),
	maplist( nilp_on_element(NilpRow, NewBasisElement, NilpIndex), Indices, NilpOnElement ),
	maplist( times(ModifyToralElement), NilpOnElement, ScaledNilpOnElement ),
	maplist( sum_mod_2, TorusOnElement, ScaledNilpOnElement, ActedOnElement ),
	maplist( times(RootVal), NewBasisElement, ScaledNewBasisElement ),
	maplist( eq, ScaledNewBasisElement, ActedOnElement).
nilp_on_element(_, _, NilpIndex, NilpIndex, NilpOnElementEntry ) :-
	#(NilpOnElementEntry) #= 0.
nilp_on_element(NilpRow, NewBasisElement, NilpIndex, Index, NilpOnElementEntry ) :-
	NilpIndex #\= Index,
	#(NewIndex) #= (NilpIndex xor Index),
	nth1( NewIndex, NewBasisElement, V),
	nth1( NewIndex, NilpRow, NilpOnElementEntry0 ),
	#(NilpOnElementEntry) #= #(NilpOnElementEntry0) * #(V).

make_new_table( Indices, Pairs, Table, NewBasis, NewTable ) :-
	maplist( set_diagonal( NewTable ), Indices ),
	maplist( intify_basis(Indices), NewBasis, IntBasis ),
	same_length( Table, NewTable ),
	maplist( same_length(NewTable), NewTable ),
	maplist( make_new_table_entry( Indices, Table, IntBasis, NewTable ), Pairs ).

make_new_table_entry( Indices, Table, IntBasis, NewTable, [I1, I2] ) :-
	nth1( I1, IntBasis, Ints1 ),
	nth1( I2, IntBasis, Ints2 ),
	#(I3) #= ( #(I1) xor #(I2) ),
	nth1( I3, IntBasis, Ints3 ),
	nth1( 1, Ints3, MainInt ),
	maplist( make_xor_pair(MainInt), Indices, XORPairs ),
	include( filter_xor_pairs( Ints1, Ints2 ), XORPairs, FilteredPairs ), 
	get_entries( Table, FilteredPairs, Entries ),
	sum( Entries, #=, #(Tot) ),
	#(Val) #= ( #(Tot) mod 2 ),
	get_entry( NewTable, [I1, I2], Val ),
	get_entry( NewTable, [I2, I1], Val ).
filter_xor_pairs( Ints1, Ints2, [I, J] ) :-
	member(I, Ints1),
	member(J, Ints2).
make_xor_pair( A, B, [C,D] ) :-
	#(C) #= #(B),
	#(D) #= ( #(A) xor #(B)).
intify_basis( Indices, BasisRow, Ints ) :-
	include( intify_basis_helper(BasisRow), Indices, Ints ).
intify_basis_helper( BasisRow, I ) :-
	nth1( I, BasisRow, 1 ).
set_diagonal( NewTable, I ) :-
	get_entry( NewTable, [I,I], 0 ).

canonical_order_tables_list( N, Roots, Powers, Tables, SwitchedTablesList, OrderedTablesList) :-
	maplist( make_table_row_sum_pair, Tables, TablesWithSums ),
	maplist( canonical_order_rows0( N, Roots, Powers, TablesWithSums ), SwitchedTablesList, OrderedTablesList ).
make_table_row_sum_pair( Table, [X, Y] ) :-
	X = Table,
	full_sorted_row_sums( Table, Y ).
canonical_order_rows0( _, _, _, _, [], OrderedTables) :-
	OrderedTables = [].
canonical_order_rows0( N, Roots, Powers, SimpleTablesWithSums, Tables, OrderedTables) :-
	length(Tables, K),
	#(K) #> 0,
	maplist(canonical_order_rows( N, Roots, Powers, SimpleTablesWithSums ), Tables, OrderedTables).
canonical_order_rows( N, Roots, Powers, SimpleTablesWithSums, Rows, OrderedRows) :-
	full_sorted_row_sums( Rows, RowSums),
	sort( RowSums, SortedRowSums),
	findall( Table, ( member(TableWithSums, SimpleTablesWithSums), TableWithSums = [ Table, RowSums ] ), SimpleTables ),
	first_sol( lex_reduce_to_simple( N, Roots, Powers, Rows, SortedRowSums ), SimpleTables, OrderedRows ).
lex_reduce_to_simple( N, Roots, Powers, Rows, SortedRowSums, Table ) :-
	can_permute_dispatcher( N, Roots, Powers, Rows, Table, SortedRowSums).

make_adjacency_mat( Tables, OrderedTablesList, AdjacencyMat ) :-
	length( Tables, M ),
	numlist( M, Indices ),
	maplist( make_adjacency_row( Tables, OrderedTablesList, Indices ), Indices, Mat ),
	transpose( Mat, TMat),
	maplist(maplist( my_max ), Mat, TMat, AdjacencyMat).
make_adjacency_row( Tables, OrderedTablesList, Indices, RowIndex, AdjacencyRow ) :-
	maplist( make_adjacency_entry( Tables, OrderedTablesList, RowIndex ), Indices, AdjacencyRow ).
make_adjacency_entry( _, _, RowIndex , RowIndex, AdjacencyEntry ) :-
	#(AdjacencyEntry) #= 0.
make_adjacency_entry( Tables, OrderedTablesList, RowIndex , ColIndex, AdjacencyEntry ) :-
	#(RowIndex) #\= #(ColIndex),
	nth1( RowIndex, Tables, Table ),
	nth1( ColIndex, OrderedTablesList, TableList ),
	set_adjacency_entry( Table, TableList, AdjacencyEntry ).
set_adjacency_entry( Table, TableList, AdjacencyEntry ) :-
	(   member( Table, TableList ),
	#(AdjacencyEntry) #= 1
	;   #(AdjacencyEntry) #= 0
	).
adjacency_mat_to_ugraph( A, G ) :-
	length( A, N ),
	numlist(N, Indices),
	maplist( collect_nbrs(A,Indices), Indices, NbrsList ), 
	maplist( make_vertex, Indices, NbrsList, G ).
make_vertex( Index, Nbrs, V ) :-
	V = Index-Nbrs.
collect_nbrs(A, Indices, Index, Nbrs ) :-
	nth1( Index, A, Row ),
	findall( X, ( member( X, Indices ), nth1( X, Row, 1 ) ), Nbrs).

conn_comps( G, Comps ) :-
	length( G, N), 
	numlist(N, Indices),
	transitive_closure(G, C),
	maplist( get_closure(C), Indices, RepeatComps  ),
	setof( Comp, member( Comp, RepeatComps ), Comps ).
get_closure(C, Index, Closure ) :-
	neighbours( Index, C, PreClosure),
	add_self( Index, PreClosure, NotSortedClosure),
	sort( NotSortedClosure, Closure).
add_self( Index, A, B) :-
	(   member(Index, A), B = A
	;   append( [A, [Index]], B )
	).
get_min_from_comp( Tables, Comp, MinTable ) :-
	get_values( Tables, Comp, CompTables ),
	min_member( my_lex, MinTable, CompTables ).
%%%%%%

%%% 5. SYMMETRY PREDICATES %%%

make_perm(Roots,Powers, Mat, RowPerm ) :-
	maplist( perm_it(Powers, Mat), Roots, RowPerm ).
perm_it(Powers, Mat, Root, Entry ) :-
	act_mat(Mat, Root, RootOut),
	bin_2_dec(Powers, RootOut, Entry).

permute_rows(Rows, PermIndex, NewRow ) :-
	nth1( PermIndex, Rows, NewRow).

can_permute_dispatcher( N, Roots, Powers, Table1, Table2, SortedRowSums ) :-
	length( Mat, N ),
	maplist(same_length(Mat), Mat),
	once(can_permute( N, Roots, Powers, Table1, Table2, Mat, SortedRowSums ) ).
can_permute( N, Roots, Powers, T1, T2, Mat, SortedRowSums ) :-
	length( Mat, N ),
	append(Mat,Ms), Ms ins 0..1,
	partition_by_row_sums( T1, SortedRowSums, Partition1 ),
	partition_by_row_sums( T2, SortedRowSums, Partition2 ),
	maplist( map_partitions(N, Mat, Powers), Partition1, Partition2 ),
	make_perm( Roots, Powers, Mat, RowPerm ),
	maplist(permute_rows(T2), RowPerm, NewRows ),
	transpose(NewRows, TNewRows),
	maplist(permute_rows(TNewRows), RowPerm, T1 ).
map_partitions( N, Mat, _, L1, L2 ) :-
	length( L1, 1 ),
	Rank is N,
	nth1( 1, L1, N1 ),
	nth1( 1, L2, N2 ),
	dec_2_bin( Rank, N1, B1 ),
	dec_2_bin( Rank, N2, B2 ),
	act_mat( Mat, B1, B2 ).
map_partitions( N, Mat, Powers, L1, L2 ) :-
	length( L1, K ),
	#(K) #> 1,
	Rank is N,
	maplist( dec_2_bin( Rank ), L1, Domain ),
	maplist( set_range( Mat, Powers, L2 ), Domain ).
set_range( Mat, Powers, Range, B1 ) :-
	act_mat( Mat, B1, B2 ),
	bin_2_dec( Powers, B2, D ),
	element( _, Range, D ).
partition_by_row_sums( Mat, SortedRowSums, Partition ) :-
	maplist( has_row_sum( Mat ), SortedRowSums, Partition ).
has_row_sum( Mat, S, L ) :-
	findall( I, ( nth1( I, Mat, R ), sum( R, #=, #(S) ) ), L ).

filter_by_row_sum_helper( RowSums, Table) :-
	full_sorted_row_sums( Table, RowSums ).
sorted_row_sums( M, Sorted ) :-
	maplist( my_sum, M, Sums ),
	sort( Sums, Sorted ).
full_sorted_row_sums( M, Sorted ) :-
	maplist( my_sum, M, Sums ),
	msort( Sums, Sorted ).
%%%%%%

%%% 6. UTILITY PREDICATES %%%
plus(X,Y,Z) :- #(Z) #= #(X) + #(Y).
eq(A,B) :- #(A) #= #(B).
notEqual(A,B) :- #(A) #\= #(B).
equal_to_zero( X ) :- #(X) #= 0.
sum_mod_2( A, B, C ) :- ((#(A) + #(B)) mod 2 ) #= #(C).
times(A,B,C) :- #(C) #= #(A) * #(B).

my_sum( Elems, Tot ) :-
	sum( Elems, #=, #(Tot) ).

my_max( A, B, C ) :-
	#(A) #< #(B),
	#(C) #= #(B).
my_max( A, B, C ) :-
	#(A) #>= #(B),
	#(C) #= #(A).

set_value_to_zero( Index, Row ) :-
	nth1( Index, Row, 0).

get_value( Row, Index, Value ) :-
	nth1( Index, Row, Value).
get_values( Row, Indices, Values ) :-
	maplist( get_value(Row), Indices, Values ).
get_entry( Rows, [RowIndex, ColIndex], Entry ) :-
	nth1( RowIndex, Rows, Row),
	nth1( ColIndex, Row, Entry).
get_entries( Rows, Indices, Entries ) :-
	maplist( get_entry(Rows), Indices, Entries ).

make_powers(N, Powers) :-
	make_powers0(N, Roots),
	reverse(Roots, Powers).
make_powers0(N, Roots) :-
	numlist(N, Indices),
	maplist(make_powers_helper, Indices, Roots).  
make_powers_helper(Ind, Root) :-
	Val is 2^(Ind - 1),
	#(Root) #= #(Val).

	dec_2_bin(Rank,N,L) :-
	dec2bin(N, L0),
	length(L0,K),
	length(L,Rank),
	M is Rank - K,
	findall(0, between(1, M, _), Zs),
	append(Zs, L0, L).
	dec2bin(0,[0]).
	dec2bin(1,[1]).
dec2bin(N,L):- 
	N > 1,
	X is N mod 2,
	Y is N // 2,  
	dec2bin(Y,L1), 
	append(L1, [X], L).
bin_2_dec(Powers, L, N) :-
	maplist(times, Powers, L, P),
	sum(P,#=,#(N)).

member_( L, E ) :- 
	member( E, L ).

get_roots(N, Roots) :-
	findall( L, ( length( L, N ), L ins 0..1, label(L) ), RootsZ ),
	RootsZ = [ _ | Roots ].
get_gln(N,Roots,GLN) :-
	length(Mat,N),
	maplist(same_length(Mat),Mat),
	append(Mat,Vs),
	Vs ins 0..1,
	findall(Mat,(mat_in_gln(Mat,Roots),label(Vs)),GLN).
mat_in_gln(Mat,Roots) :-
	maplist(vec_not_in_ker_mat(Mat),Roots).
vec_not_in_ker_mat(Mat,VecIn) :-
	same_length(VecIn,VecOut), 
	VecOut ins 0..1,
	act_mat(Mat,VecIn,VecOut),
	sum(VecOut,#>,0).
act_mat(Mat, VecIn, VecOut) :- % NB VecIn cannot have variables
	maplist(my_scalar_prod(VecIn),Mat,VecOut).
my_scalar_prod(V1,M1,X) :- 
	same_length(V1,Prods),
	maplist(times,V1,M1,Prods),
	sum(Prods,#=,#(Y)), 
	(#(Y) mod 2) #= #(X).

make_kth_row( Indices, K, KthRow ) :-
	maplist( set_kth_row( K ), Indices, KthRow ).
set_kth_row(  K, K, Entry ) :- 
	#(Entry) #= 1.
set_kth_row(  K, Index, Entry ) :- 
	#(Index) #\= K, 
	#(Entry) #= 0. 
row_of_n_zeros( N, Row ) :-
	length( Row, N ),
	maplist( equal_to_zero, Row ).
place_entry( Ind, Entry, X, Y ) :- 
	L is Ind - 1,
	length( A, L ),
	append( A, B, X ),
	append( [ A, [Entry], B ], Y ).

first_sol(Goal, List, Sol) :-
	first_sol_(List, Goal, Sol).
	first_sol_([], _, []).
first_sol_([X1|Xs1], P, Sol) :-
	(   call(P, X1) ->  Sol = X1
	;   first_sol_(Xs1, P, Sol )
	).

subset_set([], _).
subset_set([X|Xs], S) :-
	append(_, [X|S1], S),
	subset_set(Xs, S1).

numlist( N, Xs ) :-
	numlist( 1, N, Xs ).
%%%%%%
\end{lstlisting}

		\section{List of known simple Lie algebras over $\GF(2)$ up to dimension 31}\label{sec:bigtable}
\subsection{Extension of scalars and Weil restriction}
Below is a list of the simple Lie algebras over $\GF(2)$ known to the authors. Note that if $K/k$ is a field extension then any Lie algebra $L$ over a field $k$ becomes a Lie algebra $L_K$ over the field $K$ by extension of scalars: i.e. $L_K:=L\otimes_k K$ with bracket $[a\otimes \lambda,b\otimes\mu]=[a,b]\otimes\lambda\mu$ extended linearly. It need not be true that $L_K$ is simple, even when $L$ is; if $L_K$ is always simple, then we say it is \emph{absolutely simple}. Conversely, if $L$ is a Lie algebra over $K$, we get a Lie algebra $\mathrm{R}_{K/k}(L)$ by viewing $K$ as a vector space over $k$; this process is more technically known as \emph{Weil restriction}. It is easy to see that a $L$ is non-trivial and simple of dimension $n$ over $K$ implies that $\mathrm{R}_{K/k}(L)$ is simple of dimension $[K:k]\cdot n$ over $k$. (Weil restrictions of simple Lie algebras across non-trivial field extensions are never absolutely simple.)  

When listing the simple Lie algebras $\mathrm{R}_{K/k}(L_K)$ where $L$ is simple, we keep consistent with \cite{Eic10} by omitting the notation $\mathrm{R}_{K/k}$, and write just $L\otimes K$.

\subsection{Forms of Lie algebras}\label{sec:forms}We recall, say from \cite{Sel67}, that if $K/k$ is an extension of fields, then a $k$-\emph{form} of a Lie algebra $L$ over $K$ is a Lie algebra $M$ over $k$ such that $M_K\cong L$. We describe a way to make some non-trivial $k$-forms, which mirrors what happens for finite simple groups---cf.~\cite[\S13]{Car89}. The classical Lie algebras $A_n$ when $n\geq 2$, $D_n$ when $n\geq 4$ and $E_6$ admit a symmetry of their Dynkin diagram which leads to an automorphism of order $2$. In type $A_n$, i.e. when $L=\sl_{n+1}$ (or $\psl_{n+1}=\sl_{n+1}/Z(\sl_{n+1})$) then this automorphism can be described by $X\mapsto -X^T$ where $X^T$ denotes the transpose. Furthermore when $n=4$ the Dynkin diagram is a graph of order $4$ with three enges and one central node. In that case there are also automorphisms of order $3$. As was shown in \cite[Thm.~VI.6.1]{Sel67}, these give rise to non-trivial forms when $p\neq 2,3$. 

Now suppose $k=\GF(2)$ and $K/k$ is finite---thus $K=\mathbb{F}_{2^n}$ for some $n$. One can similarly produce $k$-forms of these Lie algebras which are new in some cases. So let $K=\GF(4)$ or $\GF(8)$ as required. Then one gets a $k$-form of $L$ via $M:=\{x\in L_K\mid x=\sigma(\bar x)\}$, where $\bar x$ is induced by applying the Frobenius automorphism to the coefficients of $\bar x$ written in a $k$-basis of $L$. One gets a Lie algebra of the same dimension as $L$. Thus we get Lie algebras ${}^2A_n$, ${}^2D_n$, ${}^3D_4$, ${}^2E_6$. One can show directly that some of these are isomorphic; for example $D_4\cong {}^2D_4\cong {}^3D_4$, and ${}^2A_3\cong A_3$. But many are not; for example $A_4\not\cong{}^2A_4$. By analogy with real Lie groups, one could write $\mathfrak{su}_{n+1}$ or $\mathfrak{psu}_{n+1}$ in place of ${}^2A_n$. The Lie algebra denoted `$V_8$' in \cite{Eic10} and \cite{Em22} is in fact ${}^2A_2$.

\begin{rmk} Over any field whatsoever, the only example of a simple (finite-dimensional) Lie algebra known to the authors not to be $2$-generated is the above algebra '$A_3$', or otherwise $\mathfrak{psu}_4\cong\psl_4\cong G_2$, specifically over the field of two elements. It would be interesting to know if this is genuinely the only example.\end{rmk}

\subsection{The table}Aside from those algebras which are new, we follow the naming convention in \cite{Eic10} and \cite{Em22} to which we refer the reader. The Skryabin algebra Skry is given a full treatment in \cite{GGRZ22}. The algebras DD and TW are not thin; they appear in \cite{Em22} as $L_{15,11}$ and $L_{15,10}$.

\pgfplotstableread{KnownLieAlgebras.csv}\data
\pgfplotstabletypeset[columns/Names/.style={string type},
		columns/Thin/.style={string type},
    every even row/.style={
        before row={\rowcolor[gray]{0.9}}},
    every head row/.style={
        before row=\toprule,after row=\midrule},
    every last row/.style={
        after row=\bottomrule},
        font=\footnotesize,
          row predicate/.code={%
  \pgfplotstablegetelem{#1}{Dim}\of{\data}
  \ifnum\pgfplotsretval<32\relax
  \else\pgfplotstableuserowfalse\fi}
          ]{\data}

\end{appendices}

	{\footnotesize
		\bibliographystyle{amsalpha}
		\bibliography{bib}}

\providecommand{\bysame}{\leavevmode\hbox to3em{\hrulefill}\thinspace}
\providecommand{\MR}{\relax\ifhmode\unskip\space\fi MR }
\providecommand{\MRhref}[2]{%
  \href{http://www.ams.org/mathscinet-getitem?mr=#1}{#2}
}
\providecommand{\href}[2]{#2}
\begin{thebibliography}{WSTL12}

\bibitem[Cam98]{Cam98}
Peter~J. Cameron, \emph{Introduction to algebra}, Oxford Science Publications,
  Oxford University Press, Oxford, 1998. \MR{1643302}

\bibitem[Car89]{Car89}
Roger~W. Carter, \emph{Simple groups of {L}ie type}, Wiley Classics Library,
  John Wiley \& Sons Inc., New York, 1989, Reprint of the 1972 original, A
  Wiley-Interscience Publication. \MR{MR1013112 (90g:20001)}

\bibitem[CM12]{sicstus}
Mats Carlsson and Per Mildner, \emph{S{ICS}tus {P}rolog---the first 25 years},
  Theory Pract. Log. Program. \textbf{12} (2012), no.~1-2, 35--66. \MR{2885854}

\bibitem[Col87]{colmerauer1987opening}
Alain Colmerauer, \emph{Opening the prolog iii universe: A new generation of
  prolog promises some powerful capabilities}, Byte (1987).

\bibitem[Eic10]{Eic10}
Bettina Eick, \emph{Some new simple {L}ie algebras in characteristic 2}, J.
  Symbolic Comput. \textbf{45} (2010), no.~9, 943--951. \MR{2661164}

\bibitem[EM22]{Em22}
Bettina Eick and Tobias Moede, \emph{Computing subalgebras and
  $\mathbb{Z}_2$-gradings of simple {L}ie algebras over finite fields}, 2022.

\bibitem[Fel91]{felty1991logic}
Amy Felty, \emph{A logic programming approach to implementing higher-order term
  rewriting}, International Workshop on Extensions of Logic Programming,
  Springer, 1991, pp.~135--161.

\bibitem[FH91]{FH91}
William Fulton and Joe Harris, \emph{Representation theory}, Graduate Texts in
  Mathematics, vol. 129, Springer-Verlag, New York, 1991, A first course,
  Readings in Mathematics. \MR{1153249 (93a:20069)}

\bibitem[GGRZ22]{GGRZ22}
Alexander Grishkov, Henrique Guzzo, Jr., Marina Rasskazova, and Pasha
  Zusmanovich, \emph{On simple 15-dimensional {L}ie algebras in characteristic
  2}, J. Algebra \textbf{593} (2022), 295--318. \MR{4345277}

\bibitem[GAP14]{GAP}
The~GAP group, \emph{Gap - groups, algorithms, and programming}, Version 4.7.5
  (2014), http://www.gap--system.org.

\bibitem[GS19]{GARNELO201917}
Marta Garnelo and Murray Shanahan, \emph{Reconciling deep learning with
  symbolic artificial intelligence: representing objects and relations},
  Current Opinion in Behavioral Sciences \textbf{29} (2019), 17--23, Artificial
  Intelligence.

\bibitem[Hor51]{horn_1951}
Alfred Horn, \emph{On sentences which are true of direct unions of algebras},
  Journal of Symbolic Logic \textbf{16} (1951), no.~1, 14--21.

\bibitem[Kow88]{prolog88}
Robert~A. Kowalski, \emph{The early years of logic programming}, Commun. ACM
  \textbf{31} (1988), no.~1, 38--43.

\bibitem[Lee15]{lee2015foundations}
K.D. Lee, \emph{Foundations of programming languages}, Undergraduate Topics in
  Computer Science, Springer International Publishing, 2015.

\bibitem[LF11]{Lally2011NaturalLP}
Adam Lally and Paul Fodor, \emph{Natural language processing with prolog in the
  ibm watson system}, 2011.

\bibitem[Pre89]{Pre89}
A.~A. Premet, \emph{Regular {C}artan subalgebras and nilpotent elements in
  restricted {L}ie algebras}, Mat. Sb. \textbf{180} (1989), no.~4, 542--557,
  560. \MR{997900}

\bibitem[PS02]{pereira2002prolog}
Fernando~CN Pereira and Stuart~M Shieber, \emph{Prolog and natural-language
  analysis}, Microtome Publishing, 2002.

\bibitem[PS06]{PS06}
Alexander Premet and Helmut Strade, \emph{Classification of finite dimensional
  simple {L}ie algebras in prime characteristics}, Representations of algebraic
  groups, quantum groups, and {L}ie algebras, Contemp. Math., vol. 413, Amer.
  Math. Soc., Providence, RI, 2006, pp.~185--214. \MR{2263096 (2007g:17016)}

\bibitem[Sch03]{schmid2003inductive}
U.~Schmid, \emph{Inductive synthesis of functional programs: Universal
  planning, folding of finite programs, and schema abstraction by analogical
  reasoning}, Lecture Notes in Computer Science, Springer Berlin Heidelberg,
  2003.

\bibitem[Sel67]{Sel67}
G.~B. Seligman, \emph{Modular {L}ie algebras}, Ergebnisse der Mathematik und
  ihrer Grenzgebiete, Band 40, Springer-Verlag New York, Inc., New York, 1967.
  \MR{0245627}

\bibitem[SF88]{SF88}
H.~Strade and R.~Farnsteiner, \emph{Modular {L}ie algebras and their
  representations}, Monographs and Textbooks in Pure and Applied Mathematics,
  vol. 116, Marcel Dekker Inc., New York, 1988. \MR{929682 (89h:17021)}

\bibitem[SHCO95]{somogyi1995logic}
Zoltan Somogyi, Fergus Henderson, Thomas Conway, and Richard O'Keefe,
  \emph{Logic programming for the real world}, Proceedings of the ILPS,
  vol.~95, Citeseer, 1995, pp.~83--94.

\bibitem[Skr19]{Skr19}
S.~Skryabin, \emph{The normal shapes of the symplectic and contact forms over
  algebras of divided powers}.

\bibitem[Sti88]{stickel1988prolog}
Mark~E Stickel, \emph{A prolog technology theorem prover: Implementation by an
  extended prolog compiler}, Journal of Automated reasoning \textbf{4} (1988),
  no.~4, 353--380.

\bibitem[Str04]{Str04}
H.~Strade, \emph{Simple {L}ie algebras over fields of positive characteristic.
  {I}}, de Gruyter Expositions in Mathematics, vol.~38, Walter de Gruyter \&
  Co., Berlin, 2004, Structure theory. \MR{2059133 (2005c:17025)}

\bibitem[Str09]{Str09}
Helmut Strade, \emph{Simple {L}ie algebras over fields of positive
  characteristic. {II}}, de Gruyter Expositions in Mathematics, vol.~42, Walter
  de Gruyter \& Co., Berlin, 2009, Classifying the absolute toral rank two
  case. \MR{2573283 (2011c:17035)}

\bibitem[Str13]{Str13}
\bysame, \emph{Simple {L}ie algebras over fields of positive characteristic.
  {III}}, De Gruyter Expositions in Mathematics, vol.~57, Walter de Gruyter
  GmbH \& Co. KG, Berlin, 2013, Completion of the classification. \MR{3025870}

\bibitem[Tri12]{clpfd}
Markus Triska, \emph{The finite domain constraint solver of {SWI-Prolog}},
  FLOPS, LNCS, vol. 7294, 2012, pp.~307--316.

\bibitem[WSTL12]{prolog}
Jan Wielemaker, Tom Schrijvers, Markus Triska, and Torbj\"o{}rn Lager,
  \emph{{SWI-Prolog}}, Theory and Practice of Logic Programming \textbf{12}
  (2012), no.~1-2, 67--96.

\end{thebibliography}

\end{document}